\documentclass[11pt, reqno]{amsart}
\usepackage{amsfonts,latexsym,enumerate}
\usepackage{amsmath}
\usepackage{amscd}
\usepackage{float,amsmath,amssymb,mathrsfs,bm,multirow,graphics}
\usepackage[dvips]{graphicx}
\usepackage[percent]{overpic}
\usepackage[numbers,sort&compress]{natbib}
\usepackage{xcolor}
\usepackage{todonotes}

\renewcommand{\theequation}{\mbox{\arabic{section}.\arabic{equation}}}
\newcommand{\R}{{\Bbb R}}

\newcommand{\C}{{\Bbb C}}
\newcommand{\D}{{\Bbb D}}
\newcommand{\Z}{{\Bbb Z}}

\newcommand{\diag}{\text{\upshape diag\,}}
\newcommand{\re}{\text{\upshape Re\,}}

\newcommand{\ntlim}{\lim^\angle}

\allowdisplaybreaks

\newtheorem{theorem}{Theorem}[section]
\newtheorem{proposition}[theorem]{Proposition}
\newtheorem{lemma}[theorem]{Lemma}
\newtheorem{corollary}[theorem]{Corollary}
\newtheorem{definition}[theorem]{Definition}
\newtheorem{assumption}[theorem]{Assumption}
\newtheorem{remark}[theorem]{Remark}

\newtheorem{RHproblem}[theorem]{RH problem}
\newtheorem{figuretext}{Figure}

\numberwithin{equation}{section}

\usepackage[colorlinks=true]{hyperref}
\hypersetup{urlcolor=blue, citecolor=red, linkcolor=blue}

\reversemarginpar

\input epsf
\title[The ``good'' Boussinesq equation]
{The ``good'' Boussinesq equation: \\long-time asymptotics}

\author{C. Charlier}
\address{CC: Department of Mathematics, KTH Royal Institute of Technology, 100 44 Stockholm, Sweden.}
\email{cchar@kth.se}

\author{J. Lenells}
\address{JL: Department of Mathematics, KTH Royal Institute of Technology, 100 44 Stockholm, Sweden.}
\email{jlenells@kth.se}

\author{D. Wang}
\address{DW: School of Applied Science, Beijing Information Science and Technology University, Beijing 100192, China.}
\email{wangdsh1980@163.com}

\begin{document}

\begin{abstract}
\noindent
We consider the initial-value problem for the ``good'' Boussinesq equation on the line. Using inverse scattering techniques, the solution can be expressed in terms of the solution of a $3 \times 3$-matrix Riemann-Hilbert problem. We establish formulas for the long-time asymptotics of the solution by performing a Deift-Zhou steepest descent analysis of a regularized version of this Riemann-Hilbert problem. 
\end{abstract}

\maketitle

\noindent
{\small{\sc AMS Subject Classification (2010)}: 34E05, 35G25, 35Q15, 37K15, 76B15.}

\noindent
{\small{\sc Keywords}: Asymptotics, Boussinesq equation, Riemann-Hilbert problem, inverse scattering transform, initial value problem.}

\setcounter{tocdepth}{1}
\tableofcontents

\section{Introduction}
When investigating the bidirectional propagation of small amplitude and long wavelength capillary-gravity
waves on the surface of shallow water, J. Boussinesq derived the classical Boussinesq equation \begin{equation}\label{classical-boussinesq}
 \eta_{tt}-gh_0\eta_{xx} = gh_0\bigg(\frac{3}{2}\frac{\eta^2}{h_0}+\frac{h_0^2}{3}\eta_{xx}\bigg)_{xx},
\end{equation}
where $\eta(x,t)$ is the perturbation free surface, $h_0$ is the mean depth, and $g$ is the gravitational constant \cite{Boussinesq1872}.
This equation was later rediscovered by Keulegan and Patterson \cite{KP1940}. 
In nondimensional units, equation (\ref{classical-boussinesq}) can be written as
\begin{align}\label{badboussinesq}
  u_{tt} - u_{xx} - (u^2)_{xx} - u_{xxxx} = 0,
\end{align}
where $u(x,t)$ is a real-valued function and subscripts denote partial derivatives.
Equation (\ref{badboussinesq}) is often referred to as the ``bad'' Boussinesq equation in contrast to the so-called ``good'' Boussinesq equation
\begin{align}\label{goodboussinesq}
  u_{tt} - u_{xx} + (u^2)_{xx} + u_{xxxx} = 0,
\end{align}
in which the $u_{tt}$ and $u_{xxxx}$ terms have the same sign, thus making the equation linearly well-posed (see e.g. \cite{BS1988, CT2017, F2009, HM2015, L1993} for well-posedness results for (\ref{goodboussinesq})).
Equation (\ref{goodboussinesq}) governs small nonlinear oscillations in an elastic beam and is also known as the ``nonlinear string equation'' \cite{FST1983}.

In the 1990s, Deift and Zhou proposed a steepest descent method for the asymptotic analysis of Riemann-Hilbert (RH) problems \cite{DZ1993}. The Deift-Zhou approach has been successfully utilized to determine long-time asymptotics for a large number of integrable equations such as the modified KdV equation \cite{DZ1993}, the KdV equation \cite{DVZ1994}, the nonlinear Schr\"odinger (NLS) equation \cite{DKMVZ1999,JM2013, TVZ2004}, the Camassa-Holm equation \cite{MKST2009}, the Degasperis-Procesi equation \cite{BLS2017}, and the Toda lattice \cite{DKKZ1996}.
At his $60$th birthday conference in 2005, P. Deift \cite{D2008} presented a list of sixteen open problems, among which he pointed out that ``The long-time behavior of the solutions of the Boussinesq equation with general initial data is a very interesting problem with many challenges. Even in the case with generic initial data the situation is only partially understood.'' The purpose of this paper is to take a step towards the solution of this problem. 

Following \cite{M1981, DTT1982}, we consider the following version of the ``good'' Boussinesq equation:
\begin{align}\label{boussinesq}
& u_{tt} + \frac{4}{3} (u^2)_{xx} + \frac{1}{3} u_{xxxx} = 0,
\end{align}
which can be obtained from (\ref{goodboussinesq}) by a simple shift $u \to u + 1/2$ followed by a trivial rescaling.
Our main result provides explicit formulas for the long-time asymptotics of the solution $u(x,t)$ of equation (\ref{boussinesq}) in a sector in the right half-plane $\{x>0,t>0\}$ under the assumption that the initial data lie in the Schwartz class and satisfy the physically natural assumption that $u_t(x,0)$ has zero mean.
The proof is based on a Deift-Zhou steepest descent analysis of a $3\times 3$-matrix RH problem, which is parametrized by $x$ and $t$. This RH problem was derived in \cite{CharlierLenells} by performing a spectral analysis of a Lax pair associated to (\ref{boussinesq}); it is formulated in the complex plane of the spectral parameter $k$ and has a jump contour consisting of the three lines $\R \cup \omega \R \cup \omega^2 \R$ where $\omega = e^{2\pi i/3}$. 
Although the fundamental idea behind the steepest descent analysis of this RH problem is the same in the context of other integrable equations such as the KdV and NLS equations, the analysis is severely complicated by the fact that the RH problem involves $3 \times 3$ matrices. 
Another complication stems from the fact that the RH problem associated with (\ref{boussinesq}) is singular at the origin. Therefore, instead of performing the steepest descent analysis of this RH problem directly, we will analyze a regularized version of the RH problem and then transfer the results to the singular problem.

The paper is organized as follows. The main result is stated in Section  \ref{mainresultsec}. An overview of the rather involved proof, which also contains a statement of the relevant RH problem, is presented in Section \ref{overviewsec}.
The steepest descent analysis begins in Section \ref{transsec} where several transformations of the RH problem are implemented. Local parametrices at the three critical points are constructed in Section \ref{localsec} and the resulting small-norm RH problem is estimated in Section \ref{smallnormsec}. Finally, the asymptotic behavior of $u(x,t)$ is obtained in Section \ref{uasymptoticssec}.

\section{Main result}\label{mainresultsec}
Equation (\ref{boussinesq}) can be rewritten as the system \cite{Z1974}
\begin{align}\label{boussinesqsystem}
& \begin{cases}
 v_{t} + \frac{1}{3}u_{xxx} + \frac{4}{3}(u^{2})_{x} = 0,
 \\
 u_t = v_x,
\end{cases}
\end{align}
which is equivalent to (\ref{boussinesq}) provided that $u_1(x) := u_t(x,0)$ satisfies
\begin{align}\label{u1zeromean}
\int_\R u_1(x) dx = 0.
\end{align}
Instead of analyzing (\ref{boussinesq}) with initial data $u(x,0)$ and $u_t(x,0)$ directly, we will consider the system (\ref{boussinesqsystem}) with initial data $u_0(x) = u(x,0)$ and $v_0(x) = v(x,0)$.

\subsection{Definition of $s(k)$ and $s^A(k)$}
The formulation of our main result involves two spectral functions $s(k)$ and $s^A(k)$ which are defined as follows (see \cite{CharlierLenells} for details).
Suppose $u_0(x)$ and $v_0(x)$ are real-valued functions in $\mathcal{S}(\R)$, where $\mathcal{S}(\R)$ denotes the Schwartz class of rapidly decaying functions on the real line. 
Let $\omega := e^{\frac{2\pi i}{3}}$ and let, for $j = 1,2,3$, $l_j(k) = \omega^j k$.
Define $\mathsf{U}(x,k)$ by
\begin{align}\label{mathsfUdef intro}
\mathsf{U}(x,k) = P(k)^{-1} \begin{pmatrix}
0 & 0 & 0 \\
0 & 0 & 0 \\
-v_0(x)-u_{0x}  & -2u_0(x) & 0
\end{pmatrix} P(k),
\end{align} 
where
\begin{align}\label{Pdef intro}
P(k) = \begin{pmatrix}
\omega & \omega^{2} & 1  \\
\omega^{2} k & \omega k & k \\
k^{2} & k^{2} & k^{2}
\end{pmatrix}.
\end{align}
Let $X(x,k)$ and $X^A(x,k)$ be the $3 \times 3$-matrix valued eigenfunctions defined by the linear Volterra integral equations
\begin{subequations}\label{XXAdef intro}
\begin{align}  
 & X(x,k) = I - \int_x^{\infty} e^{(x-x')\widehat{\mathcal{L}(k)}} (\mathsf{U}X)(x',k) dx',
	\\\label{XXAdefb intro}
 & X^A(x,k) = I + \int_x^{\infty} e^{-(x-x')\widehat{\mathcal{L}(k)}} (\mathsf{U}^T X^A)(x',k) dx',	
\end{align}
\end{subequations}
where $\mathcal{L} = \diag(l_1 , l_2 , l_3)$, $\widehat{\mathcal{L}}$ denotes the operator which acts on a $3 \times 3$ matrix $A$ by $\widehat{\mathcal{L}}A = [\mathcal{L}, A]$ (i.e. $e^{\widehat{\mathcal{L}}}A = e^\mathcal{L} A e^{-\mathcal{L}}$), and $\mathsf{U}^T$ denotes the transpose of $\mathsf{U}$.
The $3 \times 3$-matrix valued functions $s(k)$ and $s^A(k)$ are defined by 
\begin{align}\label{sdef intro}
& s(k) = I - \int_\R e^{-x\widehat{\mathcal{L}(k)}}(\mathsf{U}X)(x,k)dx,
 	\\ \label{sAdef intro}
& s^A(k) = I + \int_\R e^{x\widehat{\mathcal{L}(k)}}(\mathsf{U}^T X^A)(x,k)dx.
\end{align}

\subsection{Statement of the main result}
We first state our main result for the system (\ref{boussinesqsystem}); the formulation for (\ref{boussinesq}) is given as a corollary.
For simplicity, we only consider solutions in the Schwartz class.

\begin{definition}\upshape
We call $\{u(x,t), v(x,t)\}$ a {\it Schwartz class solution of (\ref{boussinesqsystem}) with initial data $u_0, v_0 \in \mathcal{S}(\R)$} if
\begin{enumerate}[$(i)$] 
  \item $u,v$ are smooth real-valued functions of $(x,t) \in \R \times [0,\infty)$.

\item $u,v$ satisfy \eqref{boussinesqsystem} for $(x,t) \in \R \times [0,\infty)$ and 
$$u(x,0) = u_0(x), \quad v(x,0) = v_0(x), \qquad x \in \R.$$ 

  \item $u,v$ have rapid decay as $|x| \to \infty$ in the sense that, for each integer $N \geq 1$ and each $T > 0$,
$$\sup_{\substack{x \in \R \\ t \in [0, T)}} \sum_{i =0}^N (1+|x|)^N(|\partial_x^i u| + |\partial_x^i v| ) < \infty.$$
\end{enumerate} 
\end{definition}

\begin{figure}
\begin{center}
 \begin{overpic}[width=.6\textwidth]{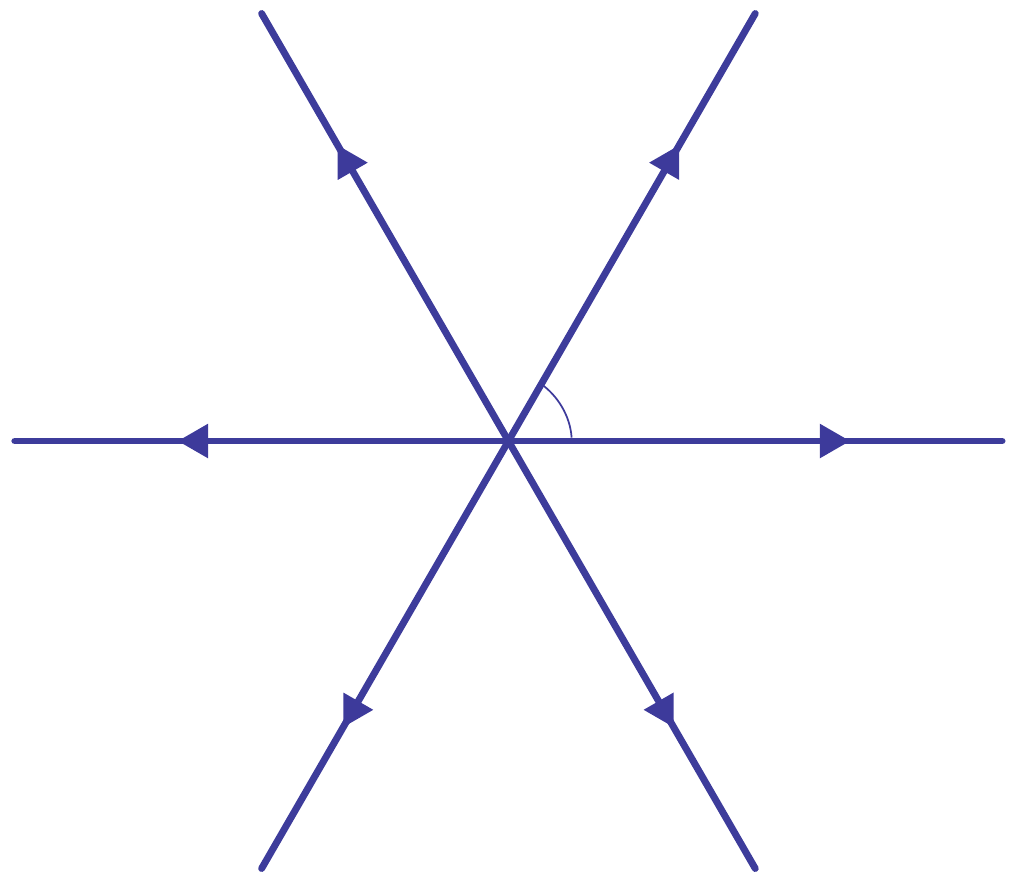}
  \put(101,42.5){\small $\Gamma$}
 \put(56,47){\small $\pi/3$}
 \put(80,60){\small $D_1$}
 \put(48,74){\small $D_2$}
 \put(17,60){\small $D_3$}
 \put(17,25){\small $D_4$}
 \put(48,12){\small $D_5$}
 \put(80,25){\small $D_6$}
  \put(81,38.7){\small $1$}
 \put(67.7,69){\small $2$}
 \put(30,69){\small $3$}
 \put(18,38.7){\small $4$}
 \put(30.5,16){\small $5$}
 \put(67.5,16){\small $6$}
   \end{overpic}
     \begin{figuretext}\label{Gamma.pdf}
       The contour $\Gamma$ and the open sets $D_n$, $n = 1, \dots, 6$, which decompose the complex $k$-plane.
     \end{figuretext}
     \end{center}
\end{figure}

Let $\{D_n\}_{n=1}^{6}$ denote the sectors shown in Figure \ref{Gamma.pdf}. 
We make the following two assumptions.

\begin{assumption}[Absence of solitons]\label{solitonlessassumption}\upshape
Assume that $(s(k))_{11}$ and $(s^A(k))_{11}$ are nonzero for $k \in \bar{D}_1\setminus \{0\}$ and $k \in \bar{D}_4\setminus \{0\}$, respectively.
\end{assumption}

\begin{assumption}[Generic behavior at $k = 0$]\label{originassumption}\upshape
Assume that
$$\lim_{k \to 0} k^2 (s(k))_{11} \neq 0, \qquad \lim_{k \to 0} k^2 (s^A(k))_{11} \neq 0.$$
\end{assumption}

Assumption \ref{solitonlessassumption} ensures that no solitons are present (the case when $s_{11}$ and $s^A_{11}$ have a finite number of simple poles off the contour can be treated by standard methods, see e.g. \cite{FI1996}, or \cite{L3x3} for a $3\times 3$ matrix case).
Assumption \ref{originassumption} ensures that $s_{11}$ and $s_{11}^A$ have double poles at $k = 0$, which is the case for generic initial data \cite{CharlierLenells}.

Define the reflection coefficient $r_1(k)$ by
\begin{align}\label{r1def}
r_1(k) = \frac{(s(k))_{12}}{(s(k))_{11}}, \qquad k \in (0,\infty).
\end{align}
If $u_0, v_0 \in \mathcal{S}(\R)$ are such that Assumptions \ref{solitonlessassumption} and \ref{originassumption} hold, then  $r_1(k)$ extends to a smooth function of $k \in [0, \infty)$ which satisfies $r_1(0) = \omega$ and which has rapid decay as $k \to \infty$, see \cite{CharlierLenells}. In particular, $|r_1(k)| < 1$ for all large enough $k> 0$ and $|r_1(0)| = 1$. Let $\zeta_0 \equiv \zeta_0(u_0, v_0) \geq 0$ be the largest nonnegative number such that $|r_1(\zeta_0/2)| = 1$, i.e.,
\begin{align}\label{zeta0def}
  \zeta_0 = 2 \max\{k \geq 0 \, | \, |r_1(k)| = 1\}.
\end{align}
We can now state our main result, which establishes the long-time behavior of $u(x,t)$ in the asymptotic sector $x/t > \zeta_0$, see Figures \ref{Ex1.pdf} and \ref{Ex2.pdf}.

\begin{figure}
\vspace{.3cm}
\begin{center}
 \begin{overpic}[width=1\textwidth]{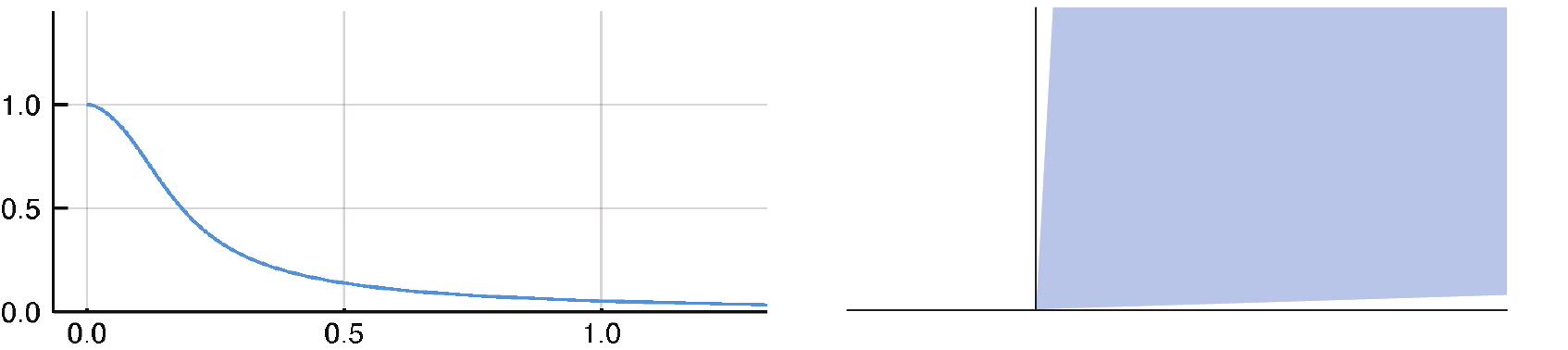}
  \put(50.5,2){\tiny $k$}
  \put(0.6,23){\tiny $|r_1(k)|$}
  \put(97.5,2){\tiny $x$}
  \put(64,0.5){\tiny $(0,0)$}
 \put(65.7,22.8){\tiny $t$}
 \put(74,13){\tiny asymptotic sector}
   \end{overpic}
     \begin{figuretext}\label{Ex1.pdf}
       Numerical example showing $|r_1(k)|$ as a function of $k \geq 0$ (left) and the corresponding asymptotic sector in the $(x,t)$-plane (right) for a choice of initial data such that $\zeta_0 = 0$.
     \end{figuretext}
     \end{center}
\end{figure}

\begin{figure}
\begin{center}
 \begin{overpic}[width=1\textwidth]{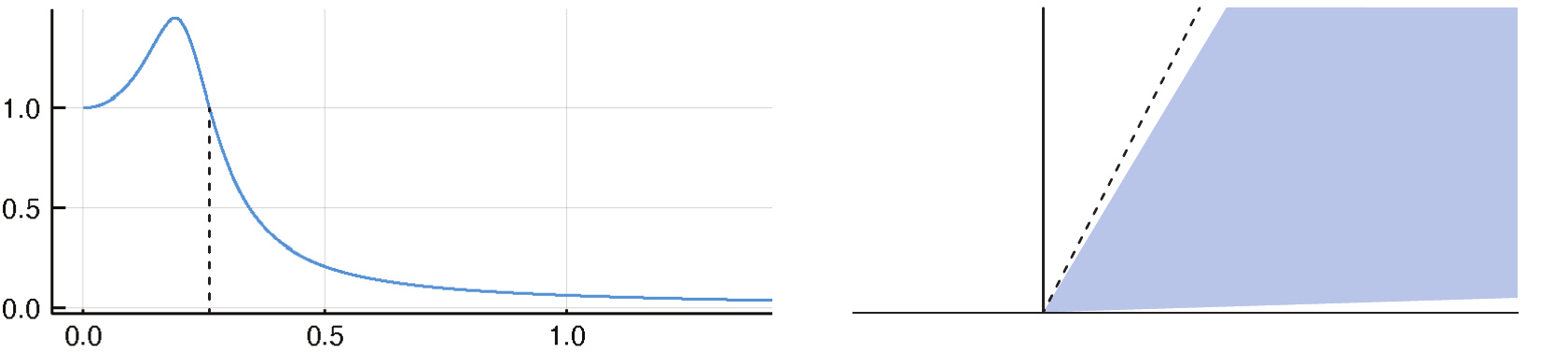}
  \put(50.5,2){\tiny $k$}
  \put(12.3,0.2){\tiny $\frac{\zeta_0}{2}$}
  \put(0.6,23.3){\tiny $|r_1(k)|$}
  \put(98,2){\tiny $x$}
  \put(64,0.5){\tiny $(0,0)$}
 \put(66.2,23){\tiny $t$}
   \put(73.2,19){\makebox(0,0){\rotatebox{63}{{\tiny $\zeta = \zeta_0$}}}}
 \put(78,13){\tiny asymptotic sector}
   \end{overpic}
     \begin{figuretext}\label{Ex2.pdf}
       Numerical example showing $|r_1(k)|$ as a function of $k \geq 0$ (left) and the corresponding asymptotic sector in the $(x,t)$-plane (right) for a choice of initial data such that $\zeta_0 > 0$.
     \end{figuretext}
     \end{center}
\end{figure}

\begin{theorem}[Long-time asymptotics for (\ref{boussinesqsystem})]\label{mainth}
Suppose $\{u(x,t), v(x,t)\}$ is a Schwartz class solution of (\ref{boussinesqsystem}) with initial data $u_0, v_0 \in \mathcal{S}(\R)$ such that Assumptions \ref{solitonlessassumption} and \ref{originassumption} hold. 
Define $\zeta_0 \geq 0$ by (\ref{zeta0def}). 
Then the following asymptotic formula holds uniformly for $\zeta = x/t$ in compact subsets of $(\zeta_0,\infty)$ as $t \to \infty$:
\begin{align}\nonumber
u(x,t) = & -\frac{3^{5/4}k_0\sqrt{\nu}}{\sqrt{2t}}  \sin\bigg(\frac{19\pi}{12}+\nu\ln\big(6\sqrt{3}tk_0^2\big) - \sqrt{3}k_0^2t
-\arg{r_1(k_0)} -\arg{\Gamma(i\nu)}
	\\ \label{uasymptotics}
& + \frac{1}{\pi }\int_{k_0}^{\infty} \ln\bigg|\frac{s-k_0}{s-\omega k_0}\bigg|\, d\ln(1 - |r_1(s)|^2)\bigg) + O\bigg(\frac{\ln t}{t}\bigg),
\end{align}
where $\Gamma$ denotes the Gamma function, $k_0 \equiv k_0(\zeta) = \zeta/2$, 
and $\nu \equiv \nu(\zeta) \geq 0$ is defined by
\begin{align*}
\nu = - \frac{1}{2\pi}\ln(1-|r_1(k_{0})|^{2}).
\end{align*}
\end{theorem}

The proof of Theorem \ref{mainth} is presented in Sections \ref{overviewsec}-\ref{uasymptoticssec}; Section \ref{overviewsec} contains an overview of the proof.

As a corollary, we obtain asymptotics of the solution of (\ref{boussinesq}) with initial data $u_0(x) = u(x,0)$ and $u_1(x) = u_t(x,0)$.

\begin{corollary}[Long-time asymptotics for (\ref{boussinesq})]\label{maincor}
Suppose $u(x,t)$ is a Schwartz class solution of the ``good'' Boussinesq equation (\ref{boussinesq}) with initial data $u_0, u_1 \in \mathcal{S}(\R)$ such that $\int_\R u_1 dx = 0$. Let $v_0(x) = \int_{-\infty}^x u_1(x') dx'$ and define $r_1:(0,\infty) \to \C$ by (\ref{r1def}). Suppose Assumptions \ref{solitonlessassumption} and \ref{originassumption} hold. Then $u$ obeys the asymptotic formula (\ref{uasymptotics}) as $t \to \infty$ uniformly for $\zeta = x/t$ in compact subsets of $(\zeta_0,\infty)$, where $\zeta_0 \geq 0$ is given by (\ref{zeta0def}).
\end{corollary}

\begin{remark}[The asymptotic sector]\upshape
Theorem \ref{mainth} provides information on the solution of (\ref{boussinesqsystem}) in the asymptotic sector $x/t > \zeta_0$, where $\zeta_0$ is defined in (\ref{zeta0def}). For a large class of initial data, it holds that $\zeta_0 = 0$ and in that case the asymptotic formula (\ref{uasymptotics}) applies to any subsector of the right half-plane $\{x > 0, t > 0\}$, see Figure \ref{Ex1.pdf}.
But there are other choices of the initial data for which $|r_1(k)| > 1$ for some $k > 0$, and then the asymptotic formula (\ref{uasymptotics}) applies only in subsectors of $x/t > \zeta_0$ where $\zeta_0 > 0$, see Figure \ref{Ex2.pdf}.
\end{remark}

\begin{remark}[Comparison with KdV]\upshape
The fact that the reflection coefficient $r_1(k)$ may have absolute value greater than $1$ is an interesting feature of the Boussinesq equation, which sets it apart from for example the KdV equation. For the KdV equation, the reflection coefficient can be defined by $r_{KdV} = s_{12}/s_{11}$ where the scattering matrix $s(k)$ satisfies $\det s = 1$, $s_{21} = \overline{s_{12}}$, and $s_{22} = \overline{s_{11}}$; thus $|r_{KdV}(k)|^2 = 1 - |s_{11}(k)|^{-2} \leq 1$ for real $k$ and equality can hold only at $k = 0$; in fact, $|r_{KdV}(0)| = 1$ for generic initial data, cf. \cite{DVZ1994}. 
The definition (\ref{r1def}) of the reflection coefficient $r_1(k)$ associated with the Boussinesq equation (\ref{boussinesq}) also satisfies $|r_1(0)| = 1$ generically. However, since $r_1(k)$ is defined in terms of the entries of a $3 \times 3$ matrix, the unit determinant condition $\det s = 1$ is, in general, not sufficient to force $|r_1| \leq 1$. In fact, it is easy to construct numerical examples for which this inequality fails, see e.g. Figure  \ref{Ex2.pdf}. 
\end{remark}

\begin{remark}[Asymptotics in the left half-plane]\label{zetanegativeremark}\upshape
In Theorem \ref{mainth}, we have, for conciseness, only presented asymptotics of $u(x,t)$ in a subsector of the right-half plane $x > 0$. A similar formula can be derived by the same methods for a subsector of the left half-plane, except that the formulas there involve $r_2 := s_{12}^A/s_{11}^A$ instead of $r_1$. Alternatively, asymptotics in the left half-plane can be obtained directly from Theorem \ref{mainth} and the invariance of the Boussinesq equation under space inversion. 
\end{remark}

\begin{remark}[Asymptotics of $v$]\upshape
Theorem \ref{mainth} provides a formula for the asymptotics of $u$. Our methods can be used to derive an analogous asymptotic formula for $v$, but since this requires somewhat lengthy estimates of $t$-derivatives (see (\ref{recoveruvn})), we have decided to not include this. 
\end{remark}

\subsection{Notation}
We summarize some notation that will be used throughout the paper. In what follows, $\gamma \subset \C$ denotes an oriented (piecewise smooth) contour.

\begin{enumerate}[$-$]

\item If $A$ is an $n \times m$ matrix, then $|A|\geq 0$ is defined by
$|A|^2 = \sum_{i,j} |A_{ij}|^2$. Note that $|A + B| \leq |A| + |B|$ and $|AB| \leq |A| |B|$.

\item $c$ and $C$ will denote generic positive constant which may change within a computation.

\item We write $\R_+ = (0, \infty)$ and $\R_- = (-\infty,0)$.

\item For $1 \leq p \leq \infty$, we write $A \in L^p(\gamma)$  if $|A|$ belongs to $L^p(\gamma)$. Then $A \in L^p(\gamma)$ iff each entry $A_{ij}$ belongs to $L^p(\gamma)$. We write $\|A\|_{L^p(\gamma)} := \| |A|\|_{L^p(\gamma)}$. 

\item We define $\dot{L}^{3}(\gamma)$ as the space of all functions $f:\gamma \to \mathbb{C}$ such that $(1+|k|)^{\frac{1}{3}}f(k) \in L^{3}(\gamma)$. If $\gamma$ is bounded, $\dot{L}^{3}(\gamma) = L^{3}(\gamma)$, but in general it only holds that $\dot{L}^{3}(\gamma) \subset L^{3}(\gamma)$. We turn $\dot{L}^{3}(\gamma)$ into a Banach space with the norm $\| f \|_{\dot{L}^{3}(\gamma)} := \| (1+|k|)^{1/3}f \|_{L^{3}(\gamma)}$.

\item We let $\dot{E}^3(\C\setminus \gamma)$ denote the space of all analytic functions $f:\C\setminus \gamma \to \C$ with the property that for each component $D$ of $\C\setminus \gamma$ there exist curves $\{C_n\}_1^\infty$ in $D$ such that the $C_n$ eventually surround each compact subset of $D$ and \begin{align*}
\sup_{n\geq 1} \int_{C_{n}} (1+|k|) \, |f(k)|^{3}|dk| < \infty.
\end{align*}

\item For a function $f$ defined in $\mathbb{C}\setminus \gamma$, we let $f_{\pm}$ denote the nontangential boundary values of $f$ from the left and right sides of $\gamma$, respectively, whenever they exist. If $f \in \dot{E}^{3}(\mathbb{C}\setminus \gamma)$, then $f_{\pm}$ exist a.e. on $\gamma$ and $f_{\pm} \in \dot{L}^{3}(\gamma)$ (see \cite[Theorem 4.1]{LenellsCarleson}).

\end{enumerate}

\section{Overview of the proof}\label{overviewsec}
The proof of Theorem \ref{mainth} consists of a Deift--Zhou steepest descent analysis of a $3 \times 3$ matrix RH problem. The jump contour $\Gamma$ of this RH problem consists of the three lines $\R \cup \omega \R \cup \omega^2 \R$, see Figure \ref{Gamma.pdf}, and the jump matrix $v$ is given explicitly in terms of $r_1(k)$ defined in (\ref{r1def}) and the function $r_2(k)$ defined by
\begin{align}\label{r2def}
r_2(k) = \frac{(s^A(k))_{12}}{(s^A(k))_{11}}, \qquad k \in (-\infty,0).
\end{align}
More precisely, $v$ is defined as follows. Define $\{l_j(k), z_j(k)\}_{j=1}^3$ by
\begin{align}
&l_j(k) = \omega^j k, \quad z_j(k) = \omega^{2j} k^{2}, \qquad k \in \C,
\end{align}
and define the complex-valued functions $\Phi_{ij}(\zeta, k)$ for $1 \leq i \neq j \leq 3$ by
\begin{align*}
\Phi_{ij}(\zeta,k) = (l_{i}-l_{j})\zeta + (z_{i}-z_{j}),
\end{align*}
where $\zeta := x/t$. By symmetry, it is enough to consider $\Phi_{21}, \Phi_{31}$, and $\Phi_{32}$, which are explicitly given by
\begin{align*}
& \Phi_{21}(\zeta,k) =  \omega(\omega-1) k (\zeta - k), \\
& \Phi_{31}(\zeta,k) = (1-\omega) k(\zeta - \omega^{2}k), \\
& \Phi_{32}(\zeta,k) = (1-\omega^{2}) k(\zeta - \omega k).
\end{align*}
Given a function $f(k)$ of $k \in \C$, we let $f^*$ denote the Schwartz conjugate of $f$, i.e.,
$$f^*(k) = \overline{f(\bar{k})}.$$
The jump matrix $v(x,t,k)$ is defined for $k \in \Gamma$ by
\begin{align}\nonumber
&  v_1 = 
  \begin{pmatrix}  
 1 & - r_1(k)e^{-t\Phi_{21}} & 0 \\
  r_1^*(k)e^{t\Phi_{21}} & 1 - |r_1(k)|^2 & 0 \\
  0 & 0 & 1
  \end{pmatrix},
\quad  v_2 = 
  \begin{pmatrix}   
 1 & 0 & 0 \\
 0 & 1 - |r_2(\omega k)|^2 & -r_2^*(\omega k)e^{-t\Phi_{32}} \\
 0 & r_2(\omega k)e^{t\Phi_{32}} & 1 
    \end{pmatrix},
   	\\ \nonumber
  &v_3 = 
  \begin{pmatrix} 
 1 - |r_1(\omega^2 k)|^2 & 0 & r_1^*(\omega^2 k)e^{-t\Phi_{31}} \\
 0 & 1 & 0 \\
 -r_1(\omega^2 k)e^{t\Phi_{31}} & 0 & 1  
  \end{pmatrix},
\quad  v_4 = 
  \begin{pmatrix}  
  1 - |r_2(k)|^2 & -r_2^*(k) e^{-t\Phi_{21}} & 0 \\
  r_2(k)e^{t\Phi_{21}} & 1 & 0 \\
  0 & 0 & 1
   \end{pmatrix},
   	\\ \label{vdef}
&  v_5 = 
  \begin{pmatrix}
  1 & 0 & 0 \\
  0 & 1 & -r_1(\omega k)e^{-t\Phi_{32}} \\
  0 & r_1^*(\omega k)e^{t\Phi_{32}} & 1 - |r_1(\omega k)|^2
  \end{pmatrix},
\quad v_6 = 
  \begin{pmatrix} 
  1 & 0 & r_2(\omega^2 k)e^{-t\Phi_{31}} \\
  0 & 1 & 0 \\
  -r_2^*(\omega^2 k)e^{t\Phi_{31}} & 0 & 1 - |r_2(\omega^2 k)|^2
   \end{pmatrix},
\end{align}
where $v_j$ denotes the restriction of $v$  to the subcontour of $\Gamma$ labeled by $j$ in Figure \ref{Gamma.pdf}.
We consider the following RH problem, which is formulated in the $L^3$-setting to ensure uniqueness (the solution of an $n \times n$-matrix $L^p$-RH problem is unique whenever it exists provided that $1 \leq n \leq p$, see \cite[Theorem 5.6]{LenellsCarleson}). 

\begin{RHproblem}[$L^{3}$-RH problem for $m$]\label{RHm}
Find a $3 \times 3$-matrix valued function $m(x,t,\cdot) \in I + \dot{E}^{3}(\mathbb{C}\setminus \Gamma)$ such that $m_+(x,t,k) = m_-(x, t, k) v(x, t, k)$ for a.e. $k \in \Gamma$.
\end{RHproblem}

By introducing the row-vector-valued function $n$ by
\begin{align}\label{ndef}
n(x,t,k) = \begin{pmatrix}\omega & \omega^2 & 1 \end{pmatrix} m(x,t,k),
\end{align}
we can transform the RH problem for $m$ into the following vector RH problem for $n$.  

\begin{RHproblem}[$L^{3}$-RH problem for $n$]\label{RHnL3}
Find a $1 \times 3$-row-vector valued function $n(x,t,\cdot) \in (\omega,\omega^{2},1) + \dot{E}^{3}(\mathbb{C}\setminus \Gamma)$ such that $n_+(x,t,k) = n_-(x, t, k) v(x, t, k)$ for a.e. $k \in \Gamma$.
\end{RHproblem}

For technical reasons, we also need the classical version of this RH problem.
\begin{RHproblem}[Classical RH problem for $n$]\label{RHnclassical}
Find a $1 \times 3$-row-vector valued function $n(x,t,k)$ with the following properties:
\begin{enumerate}[$(a)$]
\item $n(x,t,\cdot) : \C \setminus \Gamma \to \mathbb{C}^{1 \times 3}$ is analytic.

\item The limits of $n(x,t,k)$ as $k$ approaches $\Gamma \setminus \{0\}$ from the left and right exist, are continuous on $\Gamma \setminus \{0\}$, and are denoted by $n_+$ and $n_-$, respectively. Furthermore, they are related by
\begin{align}\label{njump}
  n_+(x,t,k) = n_-(x, t, k) v(x, t, k), \qquad k \in \Gamma \setminus \{0\}.
\end{align}

\item $n(x,t,k) = (\omega,\omega^{2},1) + O(k^{-1})$ as $k \to \infty$.

\item $n(x,t,k) = O(1)$ as $k \to 0$.
\end{enumerate}
\end{RHproblem}

The following result was proved in \cite{CharlierLenells}. 

\begin{proposition}\cite{CharlierLenells}\label{nprop}
Suppose the assumptions of Theorem \ref{mainth} hold. Let $U$ be an open subset of $\R \times [0,\infty)$ and suppose for each $(x,t) \in U$ that the solution of the classical RH problem \ref{RHnclassical} for $n$ is unique whenever it exists. Then RH problem \ref{RHnclassical} has a unique solution $n(x,t,k)$ for each $(x,t) \in U$ and the solution $\{u(x,t), v(x,t)\}$ of (\ref{boussinesqsystem}) can be expressed in terms of $n = (n_1, n_2, n_3)$ by
\begin{align}\label{recoveruvn}
\begin{cases}
u(x,t) = -\frac{3}{2}\frac{\partial}{\partial x}\lim_{k\to \infty}k(n_{3}(x,t,k) - 1), \\
v(x,t) = -\frac{3}{2}\frac{\partial}{\partial t}\lim_{k\to \infty}k(n_{3}(x,t,k) - 1),
\end{cases} \quad (x,t) \in U.
\end{align}
\end{proposition}
 
To use Proposition \ref{nprop}, we need the following lemma. 

\begin{lemma}\label{nlemma}
Suppose RH problem \ref{RHm} has a solution $m(x,t,\cdot)$ at some point $(x,t) \in \R \times [0,\infty)$. Then $n = (\omega, \omega^2, 1)m$ is the unique solution of RH problem \ref{RHnL3} at $(x,t)$. Moreover, if the solution of RH problem \ref{RHnclassical} exists, then it is unique and is given by $n = (\omega, \omega^2, 1)m$.
\end{lemma}
\begin{proof}
The assertion for $n = (\omega, \omega^2, 1)m$ follows as in \cite[Lemma A.5]{BLS2017}. The last claim follows because every solution of RH problem \ref{RHnclassical} is also a solution of RH problem \ref{RHnL3}.
\end{proof}

It will follow from the steepest descent analysis that RH problem \ref{RHm} has a unique solution $m$ for $t \geq T$ and $x/t$ in a compact subset of $(\zeta_{0},\infty)$. Thus Proposition \ref{nprop} and Lemma \ref{nlemma} imply that the formulas (\ref{recoveruvn}) for $u,v$ are valid for all $t \geq T$ and $x/t$ in compact subsets of $(\zeta_{0},\infty)$ if $n$ is defined by $n = (\omega, \omega^2, 1)m$.
Therefore it is enough to determine the large $t$ asymptotics of $m$.

\subsection{Steepest descent analysis}
The large $t$ behavior of $m$ can be obtained by performing a Deift--Zhou steepest descent analysis of RH problem \ref{RHm}. The first step in this analysis is to define analytic approximations of the functions $r_1$ and $r_2$ appearing in the jump matrix $v$, as well as of the combination $r_1/(1 - |r_1|^2)$. Once these approximations are in place, we can deform the contour in such a way that the new jump is close to the identity matrix everywhere except near three critical points (see Section \ref{transsec}). The critical points are the solutions of the stationary phase equations $\partial \Phi_{21}/\partial k = 0$, $\partial \Phi_{31}/\partial k = 0$, and $\partial \Phi_{32}/\partial k = 0$. For each choice of $1 \leq j < i \leq 3$, $\partial \Phi_{ij}/\partial k = 0$ has a single zero $k_{ij}$ given by
\begin{align*}
k_{21} = \frac{\zeta}{2}, \qquad k_{31} = \frac{\omega \zeta}{2}, \qquad k_{32} = \frac{\omega^{2} \zeta}{2}.
\end{align*}
Writing $k_{0} \equiv k_{21}$, these three critical points can be expressed as $k_{0}$, $\omega k_{0}$, and $\omega^{2} k_{0}$, see Figure \ref{criticalpoints2.pdf}. The signature tables for $\Phi_{21}$, $\Phi_{31}$, and $\Phi_{32}$ are shown in Figures \ref{rePhi23.pdf}-\ref{rePhi12.pdf}.
\begin{figure}
\begin{center}
\bigskip\bigskip
 \begin{overpic}[width=.55\textwidth]{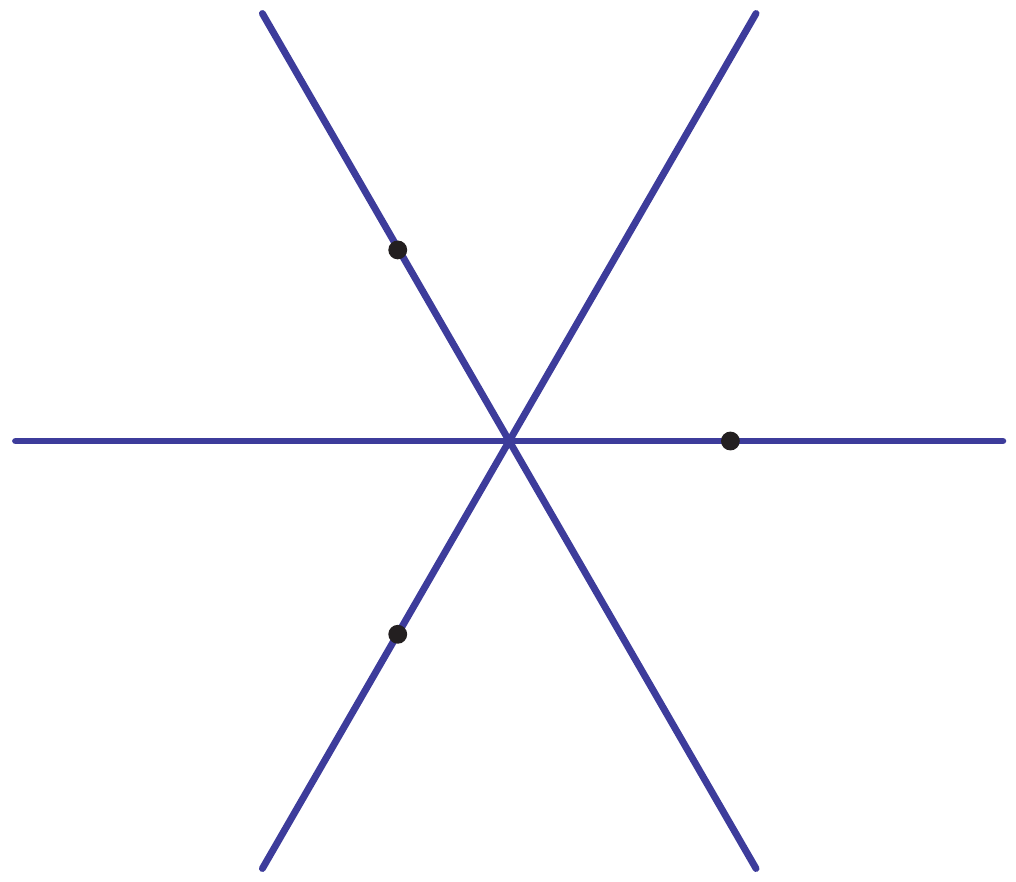}
 \put(102,42){\small $\R$}
 \put(73,89){\small $\omega^2\R$}
 \put(21,89){\small $\omega\R$}
 \put(30,61){\small $\omega k_0$}
 \put(70,38){\small $k_0$}
 \put(27,24){\small $\omega^2 k_0$}
   \end{overpic}
   \begin{figuretext}\label{criticalpoints2.pdf}
      The three critical points $k_0, \omega k_0, \omega^2 k_0$ in the complex $k$-plane for $\zeta > 0$.
      \end{figuretext}
   \end{center}
\end{figure}

\begin{figure}
\begin{center}
 \begin{overpic}[width=.5\textwidth]{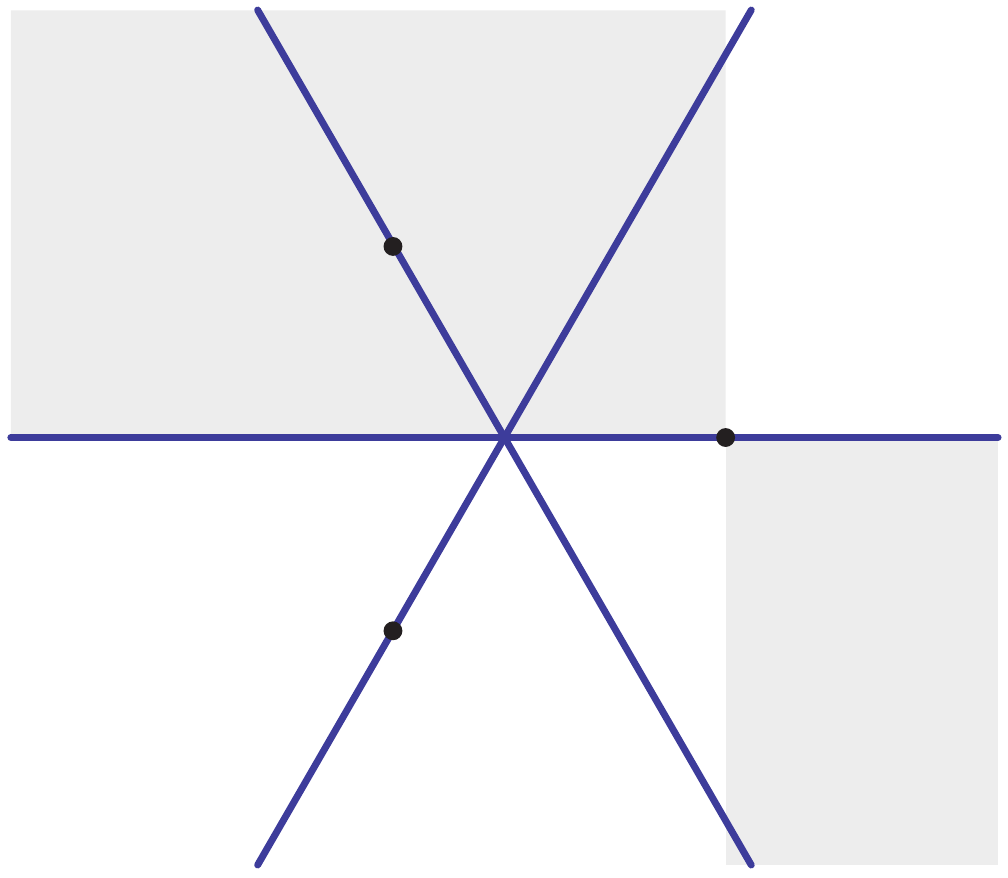}
 \put(102,42){\small $\R$}
 \put(39,74){\small $\re \Phi_{21} > 0$}
 \put(75,60){\small $\re \Phi_{21} < 0$}
 \put(30,61){\small $\omega k_0$}
 \put(70,38){\small $k_0$}
 \put(27,24){\small $\omega^2 k_0$}
   \end{overpic}
   \begin{figuretext}\label{rePhi23.pdf}
      The regions where $\re \Phi_{21} > 0$ (shaded) and $\re \Phi_{21} < 0$ (white).
      \end{figuretext}
   \end{center}
\end{figure}

\begin{figure}
\begin{center}
 \begin{overpic}[width=.5\textwidth]{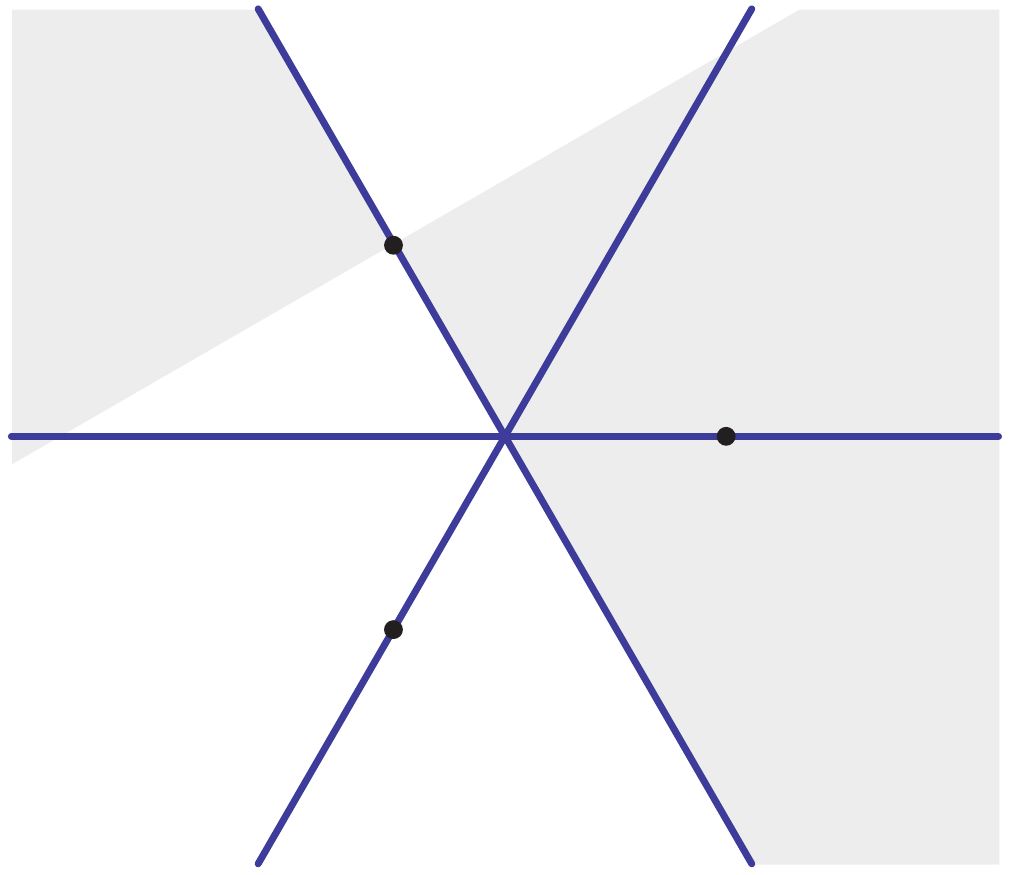}
 \put(102,42){\small $\R$}
 \put(36,78){\small $\re \Phi_{31} < 0$}
 \put(70,58){\small $\re \Phi_{31} > 0$}
  \put(30,61){\small $\omega k_0$}
 \put(70,38){\small $k_0$}
 \put(27,24){\small $\omega^2 k_0$}
   \end{overpic}
   \begin{figuretext}\label{rePhi13.pdf}
      The regions where $\re \Phi_{31} > 0$ (shaded) and $\re \Phi_{31} < 0$ (white).
      \end{figuretext}
   \end{center}
\end{figure}

Near each of the three critical points, the RH problem can be approximated by a local parametrix which is constructed in Section \ref{localsec}.
In fact, since the jump matrix $v$ obeys the symmetries
\begin{align}\label{vsymm}
v(x,t,k) = \mathcal{A} v(x,t,\omega k)\mathcal{A}^{-1}
 = \mathcal{B} \overline{v(x,t, \bar{k})}^{-1}\mathcal{B}, \qquad k \in \Gamma,
\end{align}
where
\begin{align}\label{def of Acal and Bcal}
\mathcal{A} = \begin{pmatrix}
0 & 0 & 1 \\
1 & 0 & 0 \\
0 & 1 & 0
\end{pmatrix} \qquad \mbox{ and } \qquad \mathcal{B} = \begin{pmatrix}
0 & 1 & 0 \\
1 & 0 & 0 \\
0 & 0 & 1
\end{pmatrix},
\end{align}
the solution $m$ obeys the symmetries
\begin{align}\label{msymm}
m(x,t, k) = \mathcal{A} m(x,t,\omega k)\mathcal{A}^{-1}
 = \mathcal{B} \overline{m(x,t, \bar{k})}\mathcal{B}, \qquad k \in \C \setminus \Gamma.
\end{align}
It is therefore sufficient to construct the local parametrix $m^{k_0}$ at $k_0$, because then the local parametrices at $\omega k_0$ and $\omega^2 k_0$ can be obtained by symmetry. 
In the end, we arrive at a small-norm RH problem whose solution is estimated in Section \ref{smallnormsec}. Finally, the asymptotics of $u(x,t)$ is obtained in Section \ref{uasymptoticssec}.

\begin{figure}
\begin{center}
 \begin{overpic}[width=.5\textwidth]{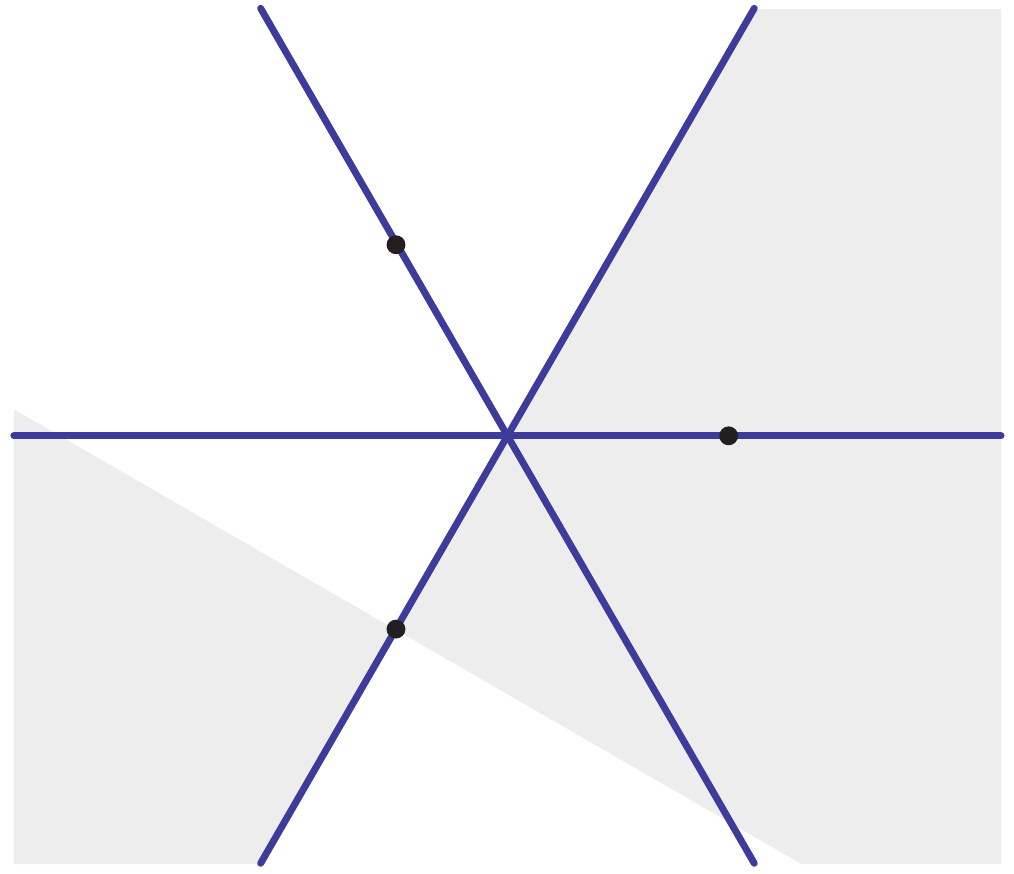}
 \put(102,41.8){\small $\R$}
 \put(38,78){\small $\re \Phi_{32} < 0$}
 \put(70,58){\small $\re \Phi_{32} > 0$}
 \put(30,61){\small $\omega k_0$}
 \put(70,38){\small $k_0$}
 \put(27,24){\small $\omega^2 k_0$}
   \end{overpic}
   \begin{figuretext}\label{rePhi12.pdf}
      The regions where $\re \Phi_{32} > 0$ (shaded) and $\re \Phi_{32} < 0$ (white).
      \end{figuretext}
   \end{center}
\end{figure}

\subsection{Assumptions for the remainder of the paper}
For the remainder of the paper, we assume that $\{u(x,t), v(x,t)\}$ is a Schwartz class solution of (\ref{boussinesqsystem}) with initial data $u_0, v_0 \in \mathcal{S}(\R)$ such that Assumptions \ref{solitonlessassumption} and \ref{originassumption} hold. We also assume that $r_1(k)$ and $r_2(k)$ are defined by (\ref{r1def}) and (\ref{r2def}) and that $\zeta_0 \geq 0$ is defined by (\ref{zeta0def}). We let $\mathcal{I}$ denote a fixed compact subset of $(\zeta_0,\infty)$.

\section{Transformations of the RH problem}\label{transsec}
By performing a number of transformations, we can bring the RH problem \ref{RHm} to a form suitable for determining the long-time asymptotics. More precisely, starting with $m$, we will define functions $m^{(j)}(x,t,k)$, $j = 1,2, 3$, such that the RH problem satisfied by $m^{(j)}$ is equivalent to the original RH problem \ref{RHm}. The RH problem for $m^{(j)}$ can be formulated as follows, where the contours $\Gamma^{(j)}$ and the jump matrices $v^{(j)}$ are specified below.

\begin{RHproblem}[RH problem for $m^{(j)}$]\label{RHmj}
Find a $3 \times 3$-matrix valued function $m^{(j)}(x,t,\cdot) \in I + \dot{E}^{3}( \C \setminus \Gamma^{(j)})$ such that $m^{(j)}_+(x,t,k) = m^{(j)}_-(x, t, k) v^{(j)}(x, t, k)$ for a.e. $k \in \Gamma^{(j)}$.
\end{RHproblem}

The jump matrix $v^{(3)}$ obtained after the third transformation has the property that it approaches the identity matrix as $t \to \infty$ everywhere on the contour except near the three critical point $\{k_0, \omega k_0, \omega^2 k_0\}$. This means that we can find the long-time asymptotics of $m^{(3)}$ by computing the contribution from three small crosses centered at these points.

The symmetries (\ref{vsymm}) and (\ref{msymm}) will be preserved at each stage of the transformations, so that, for $j = 1, 2,3$,
\begin{align}\label{Z3vjsymm}
& v^{(j)}(x,t,k) = \mathcal{A} v^{(j)}(x,t,\omega k)\mathcal{A}^{-1}
 = \mathcal{B} \overline{v^{(j)}(x,t, \bar{k})}^{-1}\mathcal{B}, \qquad k \in \Gamma^{(j)},
	\\ \label{mjsymm}
& m^{(j)}(x,t, k) = \mathcal{A} m^{(j)}(x,t,\omega k)\mathcal{A}^{-1}
 = \mathcal{B} \overline{m^{(j)}(x,t, \bar{k})}\mathcal{B}, \qquad k \in \C \setminus \Gamma^{(j)}.
\end{align}

\subsection{First transformation}
The purpose of the first transformation is to remove (except for a small remainder) the jumps across the subcontours $e^{\frac{\pi i}{3}}\R_+$, $\R_-$, and $e^{-\frac{\pi i}{3}}\R_+$ of $\Gamma$. 
To implement this transformation, we need analytic approximations of the functions $r_2^{*}$, $r_1$, and $\hat{r}_{1}^{*}$, where $\hat{r}_1(k)$ is defined by
\begin{align*}
& \hat{r}_1(k) = \frac{r_1(k)}{1 - r_1(k)r_1^{*}(k)}.
\end{align*}
We introduce open sets $U_j = U_j(\zeta) \subset \C$, $j = 1, \dots, 4$, as in Figure \ref{Ujs.pdf} such that
$$
U_1 \cup U_3=\{k \, | \, \re \Phi_{21}(\zeta, k) < 0\},\qquad
U_2 \cup U_4=\{k \, | \, \re \Phi_{21}(\zeta, k) > 0\}.$$

\begin{figure}
\begin{center}
 \begin{overpic}[width=.55\textwidth]{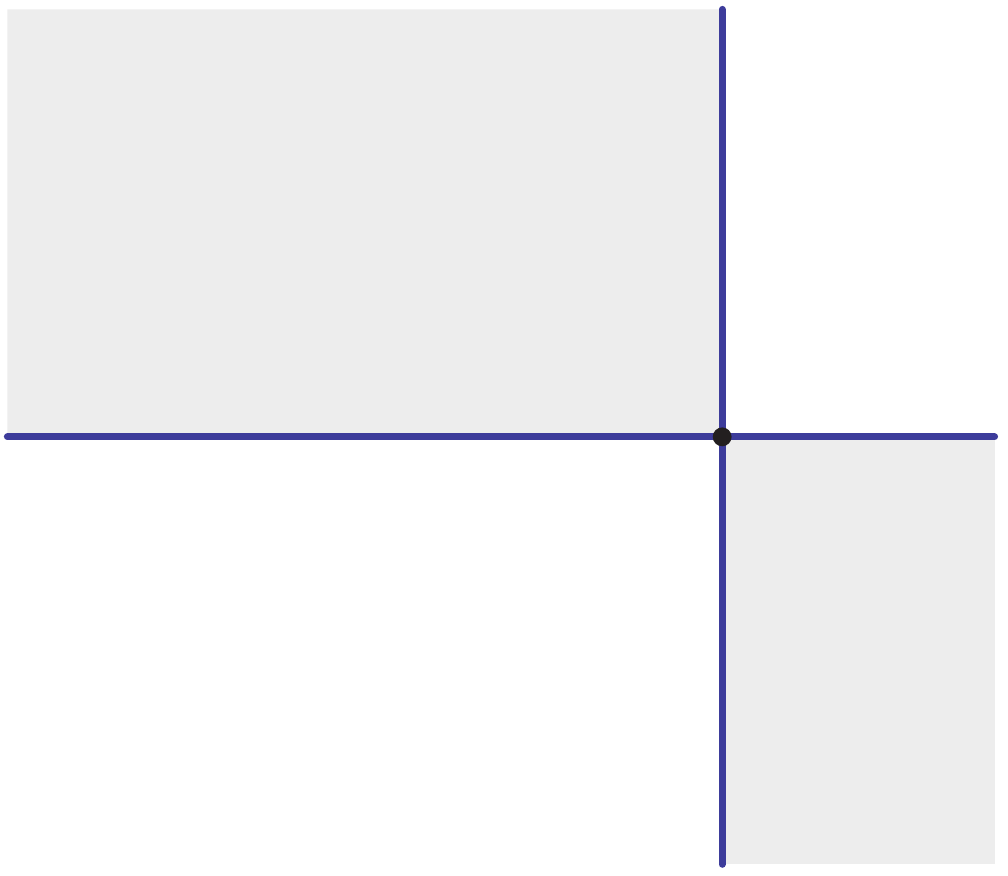}
 \put(102,42){\small $\re k$}
 \put(83,65){\small $U_1$}
 \put(34,65){\small $U_2$}
 \put(34,20){\small $U_3$}
 \put(83,20){\small $U_4$}
 \put(66,38.5){\small $k_0$}
   \end{overpic}
   \begin{figuretext}\label{Ujs.pdf}
      The open sets $\{U_j\}_1^4$ in the complex $k$-plane.
      \end{figuretext}
   \end{center}
\end{figure}

\begin{lemma}\label{decompositionlemma}
There exist decompositions
\begin{align}
& r_2^{*}(k) = r_{2,a}^{*}(x, t, k) + r_{2,r}^{*}(x, t, k), & & k \in (-\infty,0], \nonumber \\
& r_1(k) = r_{1,a}(x, t, k) + r_{1,r}(x, t, k), & & k \in [0,k_{0}], \nonumber \\
& \hat{r}_{1}^{*}(k) = \hat{r}_{1,a}^{*}(x, t, k) + \hat{r}_{1,r}^{*}(x, t, k), & & k \in [k_{0},\infty), \label{decomposition lemma analytic + remainder}
\end{align}
where the functions $r_{2,a}^{*},r_{2,r}^{*},r_{1,a},r_{1,r},\hat{r}_{1,a}^{*},\hat{r}_{1,r}^{*}$ have the following properties:
\begin{enumerate}[$(a)$]
\item 
For each $\zeta \in \mathcal{I}$ and each $t > 0$, $r_{2,a}^{*}(x, t, k)$ and $r_{1,a}(x, t, k)$ are defined and continuous for $k \in \bar{U}_2$ and analytic for $k \in U_2$, and $\hat{r}_{1,a}^{*}(x, t, k)$ is defined and continuous for $k \in \bar{U}_1$ and analytic for $k \in U_1$.

\item For each $\zeta \in \mathcal{I}$ and $t>0$, the functions $r_{2,a}^{*}$, $r_{1,a}$, and $\hat{r}_{1,a}^{*}$ satisfy
\begin{subequations}\label{bounds on analytic part}
\begin{align}
& |r_{2,a}^{*}(x, t, k)| \leq \frac{C|k-\omega k_{0}|}{1 + |k|^{2}} e^{\frac{t}{4}|\re \Phi_{21}(\zeta,k)|}, & &  k \in \bar{U}_2, \label{estimate r2astar at infty} 
	\\ 
& |\partial_{x}^{l}(r_{2,a}^{*}(x, t, k)-r_2^{*}(0))| \leq C|k| e^{\frac{t}{4}|\re \Phi_{21}(\zeta,k)|}, & &  k \in \bar{U}_2, \label{estimate r2astar at 0} 
	\\ 
& |\partial_{x}^{l}(r_{1, a}(x, t, k) - r_1(0))| \leq C |k| e^{\frac{t}{4}|\re \Phi_{21}(\zeta,k)|}, & & k \in \bar{U}_2, \label{estimate r1a at 0} 
	\\ 
& |\partial_{x}^{l}(r_{1, a}(x, t, k) - r_1(k_0))| \leq C |k - k_0| e^{\frac{t}{4}|\re \Phi_{21}(\zeta,k)|}, & & k \in \bar{U}_2, 
	 \\ 
& |\partial_{x}^{l}(\hat{r}_{1,a}^{*}(x,t,k) - \hat{r}_{1}^{*}(k_{0})) | \leq C |k - k_0| e^{\frac{t}{4}|\re \Phi_{21}(\zeta,k)|}, & & k \in \bar{U}_1, \label{estimate r1ahatstar at k0} 
	\\ 
& |\partial_{x}^{l}\hat{r}_{1, a}^{*}(x, t, k)| \leq \frac{C}{1 + |k|} e^{\frac{t}{4}|\re \Phi_{21}(\zeta,k)|}, & &  k \in \bar{U}_1,  \label{estimate r1ahatstar at infty} 
\end{align}
\end{subequations}
where $l=0,1$ and the constant $C$ is independent of $\zeta, t, k$.
\item For each $1 \leq p \leq \infty$ and $l = 0,1$, 
\begin{align}\label{norm remainder r2r}
& \text{the $L^p$-norm of $(1 + |\cdot|)\partial_x^lr_{2,r}^{*}(x,t,\cdot)$ on $(-\infty,0)$ is $O(t^{-3/2})$,}
	\\ \label{norm remainder r1r}
& \text{the $L^p$-norms of $\partial_x^l r_{1,r}(x,t,\cdot)$ and $\tfrac{r_{1,r}(x, t, \cdot)}{\cdot - k_{0}}$ on $(0, k_0)$ are $O(t^{-3/2})$,}
	\\ \label{norm remainder r1rhat}
& \text{the $L^p$-norms of $(1+|\cdot|)\partial_x^l \hat{r}_{1,r}^{*}(x, t, \cdot)$ and $\tfrac{\hat{r}_{1,r}^{*}(x, t, \cdot)}{\cdot - k_{0}}$ on $(k_0, \infty)$ are $O(t^{-3/2})$,}
\end{align}
uniformly for $\zeta \in \mathcal{I}$ as $t \to \infty$.
\end{enumerate}
\end{lemma}
\begin{proof}
The proof uses the techniques of \cite{DZ1993}. Since these techniques are rather standard by now, we omit details; see \cite[Lemma 4.8]{Lnonlinearsteepest} for a proof of a similar lemma. 
\end{proof}

In the sequel, we often write $r_{j,a}(k)$ and $r_{j,r}(k)$ instead of $r_{j,a}(x,t,k)$ and $r_{j,r}(x,t,k)$, respectively, for notational convenience. 

Recalling that $r_2=r_{2,a}+r_{2,r}$, we can factorize $v_2, v_4, v_6$ as follows:
$$v_2 = v_{2,a}^U v_{2,r} v_{2,a}^L, \qquad v_4 = v_{4,a}^U v_{4,r} v_{4,a}^L, \qquad v_6 = v_{6,a}^L v_{6,r}v_{6,a}^U,$$
where the analytic factors are given by
\begin{align*}
& v_{2,a}^U = \begin{pmatrix}   
1 & 0 & 0 \\
0 & 1 & -r_{2,a}^{*}(\omega k)e^{-t\Phi_{32}} \\
0 & 0 & 1   
\end{pmatrix},
& &
v_{2,a}^L = \begin{pmatrix}   
1 & 0 & 0 \\
0 & 1 & 0 \\
0 & r_{2,a}(\omega k)e^{t\Phi_{32}} & 1 
\end{pmatrix}, \\
& v_{4,a}^U = \begin{pmatrix}  
1 & -r_{2,a}^{*}(k)e^{-t\Phi_{21}} & 0 \\
0 & 1 & 0 \\
0 & 0 & 1 \end{pmatrix},
& &
v_{4,a}^L = \begin{pmatrix}  
1 & 0 & 0 \\
r_{2,a}(k)e^{t\Phi_{21}} & 1 & 0 \\
0 & 0 & 1 \end{pmatrix}, \\
& v_{6,a}^L =  \begin{pmatrix} 
1 & 0 & 0 \\
0 & 1 & 0 \\
-r_{2,a}^{*}(\omega^{2}k)e^{t\Phi_{31}} & 0 & 1 \end{pmatrix}, & & v_{6,a}^U =  \begin{pmatrix} 
1 & 0 & r_{2,a}(\omega^{2}k)e^{-t \Phi_{31}} \\
0 & 1 & 0 \\
0 & 0 & 1 
\end{pmatrix},
\end{align*}
and the small remainders $v_{j,r}$, $j = 2,4,6$, are given by the expressions obtained by replacing $r_j$ with $r_{j,r}$ in the definition (\ref{vdef}) of $v_j$, i.e.,
\begin{align}\nonumber
&  v_{2,r} = 
  \begin{pmatrix}   
 1 & 0 & 0 \\
 0 & 1 - r_{2,r}(\omega k)r_{2,r}^*(\omega k) & -r_{2,r}^*(\omega k)e^{-t\Phi_{32}} \\
 0 & r_{2,r}(\omega k)e^{t\Phi_{32}} & 1 
    \end{pmatrix},
   	\\ \nonumber
&  v_{4,r} = 
  \begin{pmatrix}  
  1 - |r_{2,r}(k)|^2 & -r_{2,r}^*(k) e^{-t\Phi_{21}} & 0 \\
  r_{2,r}(k)e^{t\Phi_{21}} & 1 & 0 \\
  0 & 0 & 1
   \end{pmatrix},
   	\\ \nonumber
&  v_{6,r} = 
  \begin{pmatrix} 
  1 & 0 & r_{2,r}(\omega^2 k)e^{-t\Phi_{31}} \\
  0 & 1 & 0 \\
  -r_{2,r}^*(\omega^2 k)e^{t\Phi_{31}} & 0 & 1 - r_{2,r}(\omega^2 k)r_{2,r}^*(\omega^2 k)
   \end{pmatrix}.
\end{align}
Define the sectionally analytic function $m^{(1)}$ by
 $$m^{(1)}(x,t,k) = m(x,t,k)G(x,t,k),$$
 where $G$ is defined by
\begin{align}\label{Gdef}
G(x,t,k) = \begin{cases} v_{2,a}^U, & k \in D_1, \\
 (v_{2,a}^L)^{-1}, & k \in D_2, \\
v_{4,a}^U , & k \in D_3, \\
(v_{4,a}^L)^{-1}, & k \in D_4, \\
v_{6,a}^L, & k \in D_5, \\
(v_{6,a}^U)^{-1}, & k \in D_6.
\end{cases}
\end{align}

\begin{lemma}\label{Glemma}
$G(x,t,k)$ and $G(x,t,k)^{-1}$ are uniformly bounded for $k \in \mathbb{C}\setminus \Gamma$, $t> 0$, and $\zeta \in \mathcal{I}$. Moreover, $G = I + O(k^{-1})$ as $k \to \infty$. 
\end{lemma}
\begin{proof}
We have $\re \Phi_{21}(\zeta,k) > 0$ for $k \in D_{3}$ (see Figure \ref{rePhi23.pdf}). Therefore, by virtue of \eqref{estimate r2astar at infty},
\begin{align*}
|v_{4,a}^{U}(x,t,k)-I| \leq \frac{C}{1+|k|}e^{-c t |\re \Phi_{21}(\zeta,k)|}, \qquad k \in D_{3},
\end{align*}
uniformly for $\zeta \in \mathcal{I}$.
Since $\re \Phi_{21}(\zeta,k) < 0$ for $\zeta \in D_{4}$ (see Figure \ref{rePhi23.pdf} again), we deduce similarly that
\begin{align*}
& |r_{2,a}(x, t, k)| \leq \frac{C}{1 + |k|} e^{\frac{t}{4}|\re \Phi_{21}(\zeta,k)|}, \qquad k \in \bar{U}_3,
\end{align*}
and hence
\begin{align*}
|(v_{4,a}^{L})^{-1}(x,t,k)-I| \leq \frac{C}{1+|k|}e^{-c t |\re \Phi_{21}(\zeta,k)|}, \qquad k \in D_{4}.
\end{align*}
We appeal to the $\mathcal
A$-symmetry of \eqref{Z3vjsymm} to extend these bounds to the other sectors. 
\end{proof}

It follows from Lemma \ref{Glemma} that $m$ satisfies RH problem \ref{RHm} if and only if $m^{(1)}$
satisfies RH problem \ref{RHmj} with $j = 1$, where $\Gamma^{(1)} = \Gamma$ and the jump matrix $v^{(1)}$ is given on $\Gamma_1 \cup \Gamma_3 \cup \Gamma_5$ by
$$v_1^{(1)} = v_{6,a}^U v_1 v_{2,a}^U, \qquad
v_3^{(1)} = v_{2,a}^L v_3 v_{4,a}^U, \qquad
v_5^{(1)} = v_{4,a}^L v_5 v_{6,a}^L,$$
and the small jumps remaining on  $\Gamma_2 \cup \Gamma_4 \cup \Gamma_6$ are given by
$$v_j^{(1)} = v_{j,r}, \qquad j = 2,4,6.$$
Here $\Gamma_j$ denotes the subcontour of $\Gamma$ labeled by $j$ in Figure \ref{Gamma.pdf}. 
More explicitly, the jump matrices $v_j^{(1)}$, $j = 1,3,5$, can be expressed as
\begin{align*}
& v_1^{(1)} = \begin{pmatrix}
1 & -r_1(k)e^{-t \Phi_{21}} & \beta(k)e^{-t \Phi_{31}} \\
r_1^{*}(k)e^{t \Phi_{21}} & 1-r_1(k)r_1^{*}(k) & \alpha(k)e^{-t\Phi_{32}} \\
0 & 0 & 1
\end{pmatrix}, \\
& v_{3}^{(1)} = \begin{pmatrix}
1 - r_1(\omega^{2} k)r_1^{*}(\omega^{2} k) & \alpha(\omega^{2} k) e^{-t \Phi_{21}} & r_1^{*}(\omega^{2} k) e^{-t \Phi_{31}} \\
0 & 1 & 0 \\
-r_1(\omega^{2} k)e^{t \Phi_{31}} & \beta(\omega^{2}k)e^{t\Phi_{32}} & 1
\end{pmatrix}, \\
& v_{5}^{(1)} = \begin{pmatrix}
1 & 0 & 0 \\
\beta(\omega k) e^{t \Phi_{21}} & 1 & -r_1(\omega k) e^{-t\Phi_{32}} \\
\alpha(\omega k)e^{t\Phi_{31}} & r_1^{*}(\omega k) e^{t\Phi_{32}} & 1-r_1(\omega k)r_1^{*}(\omega k)
\end{pmatrix},
\end{align*}
where the functions $\alpha(k) \equiv \alpha(x,t,k)$ and $\beta(k) \equiv \beta(x,t,k)$ are defined by
\begin{align*}
& \alpha(k) = -r_{2,a}^{*}(\omega k)(1 - r_1(k)r_1^*(k)), & & k\in \R_+,
	\\
& \beta(k) =r_{2,a}(\omega^2 k) + r_1(k)r_{2,a}^{*}(\omega k), & & k \in \R_+.
\end{align*}

\begin{figure}
\begin{center}
 \begin{overpic}[width=.7\textwidth]{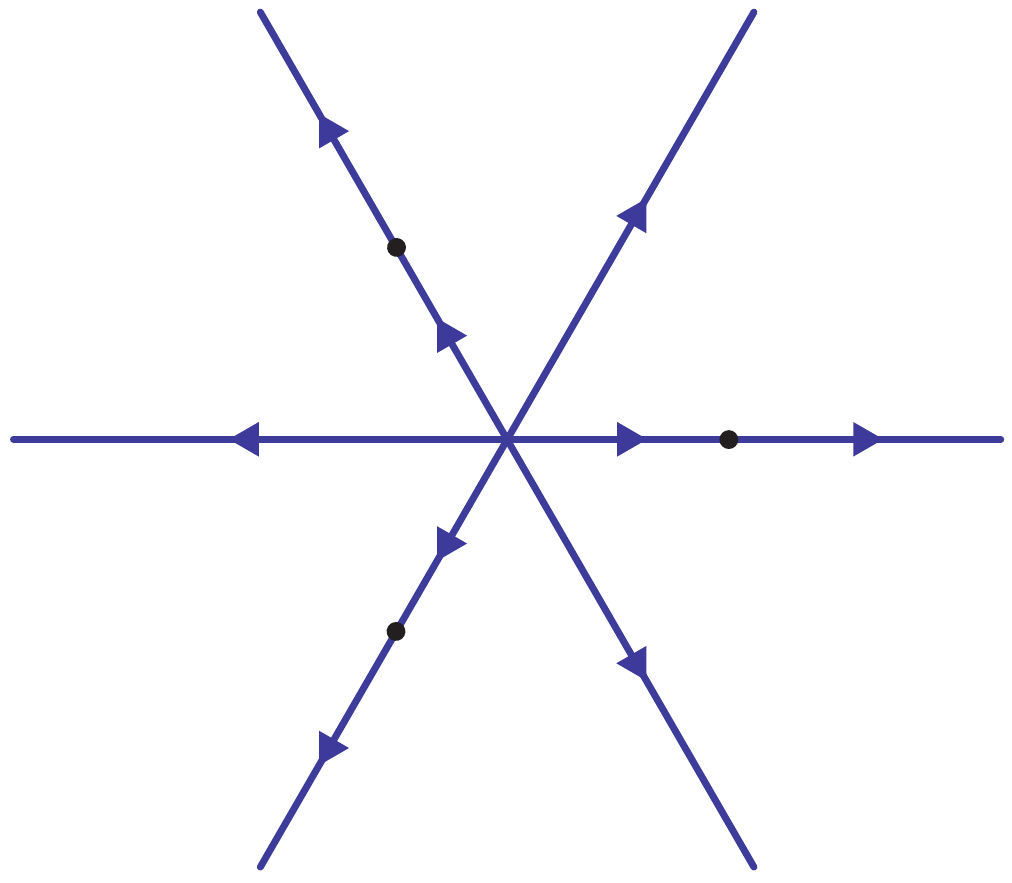}
  \put(101,42.5){\small $\Gamma^{(2)}$}
 \put(84,39){\small $1$}
 \put(65,63){\small $2$}
 \put(35,73){\small $3$}
 \put(24,39){\small $4$}
 \put(35,12){\small $5$}
 \put(65,21){\small $6$}
 \put(60,39){\small $7$}
 \put(46,53){\small $8$}
 \put(46,32){\small $9$}
 \put(72,39){\small $k_0$}
 \put(41,62){\small $\omega k_0$}
 \put(41,22){\small $\omega^2 k_0$}
   \end{overpic}
     \begin{figuretext}\label{Gamma2.pdf}
       The contour $\Gamma^{(2)}$ in the complex $k$-plane.
     \end{figuretext}
     \end{center}
\end{figure}

\subsection{Second transformation}
Let $\Gamma^{(2)} = \cup_{j=1}^9 \Gamma_j^{(2)}$ denote the contour displayed in Figure \ref{Gamma2.pdf}, where $\Gamma_1^{(2)} = [k_0, \infty)$ etc. For each $\zeta \in \mathcal{I}$, we choose $\delta_1(\zeta, k)$ such that $\delta_1$ is analytic except for the following jump across $\Gamma_1^{(2)}$:
\begin{align*}
\delta_{1+}(\zeta, k) = \delta_{1-}(\zeta, k)(1 - |r_1(k)|^2), \qquad k \in \Gamma_1^{(2)},
\end{align*}
and such that
\begin{align}\label{delta1 asymp at inf}
\delta_1(\zeta, k) = 1 + O(k^{-1}), \qquad k \to \infty.
\end{align} 
The definition (\ref{zeta0def}) of $\zeta_0$ together with the relation $k_{0}=\zeta/2$ implies that there exists an $\epsilon > 0$ such that 
\begin{align}\label{rk0leq1minusepsilon}
|r(k_0)| \leq 1 - \epsilon \quad \text{for all $k \in [k_{0},\infty)$ and all $\zeta \in \mathcal{I}$}.
\end{align}
Hence, by the Plemelj formulas, we find
\begin{equation} \label{delta1def}
\delta_1(\zeta, k) = \exp \left\{ \frac{1}{2\pi i} \int_{[k_0, \infty)} \frac{\ln(1 - |r_1(s)|^2)}{s - k} ds \right\}, \qquad k \in \C \setminus \Gamma_1^{(2)}.
\end{equation}

Let $\ln_0(k)$ denote the logarithm of $k$ with branch cut along $\arg k=0$, i.e., $\ln_0(k) = \ln|k| + i \arg_0 k$ with $\arg_0 k\in(0,2\pi)$. 

\begin{lemma}\label{deltalemma}
The function $\delta_1(\zeta, k)$ has the following properties:
\begin{enumerate}[$(a)$]
\item $\delta_{1}$ can be written as
\begin{align}\label{delta1 expression in terms of log and chi1}
\delta_1(\zeta,k) = e^{-i  \nu \ln_{0}(k-k_{0})}e^{-\chi_{1}(\zeta,k)},
\end{align}
where $\nu \equiv \nu(\zeta) \geq 0$ is defined by
\begin{align*}
\nu = - \frac{1}{2\pi}\ln(1-|r_1(k_{0})|^{2}), \qquad \zeta \in \mathcal{I},
\end{align*}
and 
\begin{align}\label{def of chi1}
& \chi_{1}(\zeta,k) = \frac{1}{2\pi i} \int_{k_{0}}^{\infty}  \ln_{0}(k-s) d\ln(1-|r_1(s)|^{2}).
\end{align}
\item For each $\zeta \in \mathcal{I}$, $\delta_1(\zeta, k)$ and $\delta_1(\zeta, k)^{-1}$ are analytic functions of $k \in \C \setminus \Gamma_1^{(2)}$ with continuous boundary values on $\Gamma_1^{(2)}\setminus \{k_0\}$. Moreover,
\begin{align}\label{delta1bound}
\sup_{\zeta \in \mathcal{I}} \sup_{k \in \C \setminus \Gamma_1^{(2)}} |\delta_1(\zeta,k)^{\pm 1}| < \infty.
\end{align}
\item $\delta_1$ obeys the symmetry
\begin{align}\label{symmetry of delta1}
\delta_1(\zeta, k) = \overline{\delta_1(\zeta, \bar{k})}^{-1}, \qquad \zeta \in \mathcal{I}, \ k \in \C \setminus \Gamma_1^{(2)}.
\end{align}
\item As $k \to k_{0}$ along a path which is nontangential to $(k_{0},\infty)$, we have
\begin{align}
& |\chi_{1}(\zeta,k)-\chi_{1}(\zeta,k_{0})| \leq C |k-k_{0}|(1+|\ln|k-k_{0}||), \label{asymp chi1 at k0} \\
& |\partial_{x}(\chi_{1}(\zeta,k)-\chi_{1}(\zeta,k_{0}))| \leq \frac{C}{t}(1+|\ln|k-k_{0}||), \label{asymp chi1 der at k0}
\end{align}
where $C$ is independent of $\zeta \in \mathcal{I}$. Furthermore,
\begin{align}\label{estimate der of chi}
|\partial_{x}\chi_{1}(\zeta,k_{0})| & = \frac{1}{t}\Big|\partial_{u}\chi_{1}(u,v)\big|_{(u,v) = (\zeta,k_{0})} + \frac{1}{2} \partial_{v}\chi_{1}(u,v)\big|_{(u,v) = (\zeta,k_{0})}\Big| \leq \frac{C}{t}
\end{align}
and
\begin{align}\label{der x of delta1}
\partial_{x} (\delta_{1}(\zeta,k)^{\pm 1}) = \frac{\pm i \nu}{2t(k-k_{0})} \delta_{1}(\zeta,k)^{\pm 1}.
\end{align}
\end{enumerate}
\end{lemma}
\begin{proof}
The lemma follows from (\ref{delta1def}) and relatively straightforward estimates.
\end{proof}

The functions $\delta_{3}$ and $\delta_{5}$ defined by
\begin{align*}
& \delta_3(\zeta,k) = \delta_1(\zeta,\omega^{2}k), & & k \in \mathbb{C}\setminus \Gamma_{3}^{(2)}, \\
& \delta_5(\zeta,k) = \delta_1(\zeta,\omega k), & & k \in \mathbb{C}\setminus \Gamma_{5}^{(2)},
\end{align*}
satisfy the jump relations
\begin{align*}
& \delta_{3+}(\zeta,k) = \delta_{3-}(\zeta,k)(1 - |r_1(\omega^{2}k)|^2), & & k \in \Gamma_{3}^{(2)}, \\
& \delta_{5+}(\zeta,k) = \delta_{5-}(\zeta,k)(1 - |r_1(\omega k)|^2), & & k \in \Gamma_{5}^{(2)}.
\end{align*}
The jump matrix $v^{(1)}$ cannot be appropriately factorized on the subcontour $\Gamma_1^{(2)} \cup \Gamma_3^{(2)} \cup \Gamma_5^{(2)}$ of $\Gamma^{(2)}$. Hence we introduce $m^{(2)}$ by
$$m^{(2)}(x,t,k) = m^{(1)}(x,t,k) \Delta(\zeta, k),$$
where the $3\times 3$-matrix valued function $\Delta(\zeta,k)$ is defined by
\begin{align}\label{Deltadef}
\Delta(\zeta, k)
= \begin{pmatrix} 
\frac{\delta_1(\zeta, k)}{\delta_3(\zeta, k)} & 0 & 0 \\
0 & \frac{\delta_5(\zeta, k)}{\delta_1(\zeta, k)} & 0 \\
0 & 0 & \frac{\delta_3(\zeta, k)}{\delta_5(\zeta, k)} \end{pmatrix}.
\end{align}
From \eqref{delta1bound} and \eqref{delta1 asymp at inf}, we infer that $\Delta$ and $\Delta^{-1}$ are uniformly bounded for $\zeta \in \mathcal{I}$ and $k \in \mathbb{C}\setminus (\Gamma_1^{(2)}\cup \Gamma_{3}^{(2)} \cup \Gamma_{5}^{(2)})$ and that
\begin{align}\label{Deltaatinfty}
\Delta(\zeta, k) = I + O(k^{-1}) \quad \text{as $k \to \infty$}.
\end{align}
It follows that $m$ satisfies RH problem \ref{RHm} if and only if $m^{(2)}$
satisfies RH problem \ref{RHmj} with $j = 2$, where the jump matrix $v^{(2)}$ is given by $v^{(2)} =  \Delta_-^{-1}v^{(1)} \Delta_+$. A computation gives
\begin{align*}
v_1^{(2)} & = \begin{pmatrix}
\frac{\delta_{1+}}{\delta_{1-}} & -\frac{\delta_{3}\delta_{5}}{\delta_{1-}\delta_{1+}} r_1(k)e^{-t \Phi_{21}} & \frac{\delta_{3}^{2}}{\delta_{1-}\delta_{5}}\beta(k)e^{-t \Phi_{31}} \\
\frac{\delta_{1-} \delta_{1+}}{\delta_{3} \delta_{5}} r_1^{*}(k)e^{t\Phi_{21}} & \frac{\delta_{1-}}{\delta_{1+}}(1-r_1(k)r_1^{*}(k)) & \frac{\delta_{1-}\delta_{3}}{\delta_{5}^{2}}\alpha(k) e^{-t\Phi_{32}} \\
0 & 0 & 1 
\end{pmatrix} \\
& = \begin{pmatrix}
1 - r_1(k)r_1^{*}(k) & - \frac{\delta_{3} \delta_{5}}{\delta_{1-}^{2}} \frac{r_1(k)}{1-r_1(k)r_1^{*}(k)}e^{-t\Phi_{21}} & \frac{\delta_{3}^{2}}{\delta_{1-}\delta_{5}} \beta(k)e^{-t\Phi_{31}} \\
\frac{\delta_{1+}^{2}}{\delta_{3}\delta_{5}} \frac{r_1^{*}(k)}{1-r_1(k)r_1^{*}(k)}e^{t\Phi_{21}} & 1 & -r_{2,a}^{*}(\omega k) \frac{\delta_{1+}\delta_{3}}{\delta_{5}^{2}}e^{-t \Phi_{32}} \\
0 & 0 & 1
\end{pmatrix},
	\\
v_2^{(2)} & = \begin{pmatrix}
1 & 0 & 0 \\
0 & 1-r_{2,r}(\omega k)r_{2,r}^{*}(\omega k) & - \frac{\delta_{1}  \delta_{3}}{\delta_{5}^{2}} r_{2,r}^{*}(\omega k)e^{-t\Phi_{32}} \\
0 & \frac{\delta_{5}^{2}}{\delta_{1}\delta_{3}} r_{2,r}(\omega k)e^{t\Phi_{32}} & 1
\end{pmatrix} ,
 	\\
v_7^{(2)} & = \begin{pmatrix}
1 & - \frac{\delta_{3}\delta_{5}}{\delta_{1}^{2}} r_1(k)e^{-t\Phi_{21}} & \frac{\delta_{3}^{2}}{\delta_{1}\delta_{5}} \beta(k)e^{-t\Phi_{31}} \\
\frac{\delta_{1}^{2}}{\delta_{3}\delta_{5}} r_1^{*}(k)e^{t\Phi_{21}} & 1 - r_1(k)r_1^{*}(k) & \frac{\delta_{1}\delta_{3}}{\delta_{5}^{2}} \alpha(k) e^{-t\Phi_{32}} \\
0 & 0 & 1
\end{pmatrix}.
\end{align*}
The remaining jumps $v_j^{(2)}$ can be obtained from these matrices together with the $\Z_3$ symmetry (\ref{Z3vjsymm}) and are given by
\begin{align*}
& v_3^{(2)} = \begin{pmatrix}
1 & -\frac{\delta_{3+}\delta_{5}}{\delta_{1}^{2}} r_{2,a}^{*}(k)e^{-t\Phi_{21}} & \frac{\delta_{3+}^{2}}{\delta_{1}\delta_{5}} \frac{r_1^{*}(\omega^{2}k)e^{-t\Phi_{31}}}{1-r_1(\omega^{2}k)r_1^{*}(\omega^{2}k)} \\
0 & 1 & 0 \\
- \frac{\delta_{1}\delta_{5}}{\delta_{3-}^{2}} \frac{r_1(\omega^{2}k)e^{t\Phi_{31}}}{1-r_1(\omega^{2}k)r_1^{*}(\omega^{2} k)} & \frac{\delta_{5}^{2}}{\delta_{1} \delta_{3-}} \beta(\omega^{2}k)e^{t \Phi_{32}} & 1-r_1(\omega^{2} k)r_1^{*}(\omega^{2}k)
\end{pmatrix},
 	\\
& v_4^{(2)} = \begin{pmatrix}
1 - r_{2,r}(k)r_{2,r}^{*}(k) & - \frac{\delta_{3}\delta_{5}}{\delta_{1}^{2}} r_{2,r}^{*}(k)e^{-t\Phi_{21}} & 0 \\
\frac{\delta_{1}^{2}}{\delta_{3}\delta_{5}} r_{2,r}(k)e^{t\Phi_{21}} & 1 & 0 \\
0 & 0 & 1
\end{pmatrix},
   	\\
& v_5^{(2)} = \begin{pmatrix}
1 & 0 & 0 \\
\frac{\delta_{1}^{2}}{\delta_{3} \delta_{5-}} \beta(\omega k)e^{t \Phi_{21}} & 1-r_1(\omega k) r_1^{*}(\omega k) & -\frac{\delta_{3}\delta_{1}}{\delta_{5-}^{2}} \frac{r_1(\omega k)e^{-t\Phi_{32}}}{1-r_1(\omega k)r_1^{*}(\omega k)} \\
-\frac{\delta_{1} \delta_{5+}}{\delta_{3}^{2}} r_{2,a}^{*}(\omega^{2}k)e^{t\Phi_{31}} & \frac{\delta_{5+}^{2}}{\delta_{1} \delta_{3}} \frac{r_1^{*}(\omega k)e^{t\Phi_{32}}}{1-r_1(\omega k)r_1^{*}(\omega k)} & 1
\end{pmatrix} ,
   	\\
& v_6^{(2)} = \begin{pmatrix}
1 & 0 & \frac{\delta_{3}^{2}}{\delta_{1}\delta_{5}} r_{2,r}(\omega^{2}k)e^{-t\Phi_{31}} \\
0 & 1 & 0 \\
- \frac{\delta_{1}\delta_{5}}{\delta_{3}^{2}} r_{2,r}^{*}(\omega^{2}k)e^{t\Phi_{31}} & 0 & 1-r_{2,r}(\omega^{2}k)r_{2,r}^{*}(\omega^{2}k)
\end{pmatrix},
	\\
& v_8^{(2)} = \begin{pmatrix}
1-r_1(\omega^{2}k)r_1^{*}(\omega^{2}k) & \frac{\delta_{3} \delta_{5}}{\delta_{1}^{2}} \alpha(\omega^{2} k) e^{-t\Phi_{21}} & \frac{\delta_{3}^{2}}{\delta_{1}\delta_{5}} r_1^{*}(\omega^{2}k)e^{-t\Phi_{31}} \\
0 & 1 & 0 \\
- \frac{\delta_{1} \delta_{5}}{\delta_{3}^{2}} r_1(\omega^{2}k)e^{t\Phi_{31}} & \frac{\delta_{5}^{2}}{\delta_{1}\delta_{3}} \beta(\omega^{2}k)e^{t \Phi_{32}} & 1
\end{pmatrix},
	\\
& v_9^{(2)}= \begin{pmatrix}
1 & 0 & 0 \\
\frac{\delta_{1}^{2}}{\delta_{3}\delta_{5}} \beta(\omega k)e^{t\Phi_{21}} & 1 & - \frac{\delta_{1}\delta_{3}}{\delta_{5}^{2}} r_1(\omega k) e^{-t \Phi_{32}} \\
\frac{\delta_{1}\delta_{5}}{\delta_{3}^{2}} \alpha(\omega k)e^{t \Phi_{31}} & \frac{\delta_{5}^{2}}{\delta_{1}\delta_{3}} r_1^{*} (\omega k) e^{t \Phi_{32}} & 1 - r_1(\omega k)r_1^{*}(\omega k) 
\end{pmatrix}.
 \end{align*}

\subsection{Third transformation} \label{section4.5}
The $(11)$-entry of $v_{1}^{(2)}$ can be rewritten as follows:
\begin{align*}
(v_{1}^{(2)})_{11} = 1- r_1(k)r_1^{*}(k) = 1- \frac{\delta_{1+}}{\delta_{1-}}\hat{r}_{1}(k)\frac{\delta_{1+}}{\delta_{1-}}\hat{r}_{1}^{*}(k).
\end{align*}
Therefore, using the general identity
$$
\begin{pmatrix}
1+f_{1}f_{3} & f_{1} & f_{2} \\
f_{3} & 1 & f_{4} \\
0 & 0 & 1
\end{pmatrix} = \begin{pmatrix}
1 & f_{1,a} & f_{2}-f_{1}f_{4} \\
0 & 1 & 0 \\
0 & 0 & 1
\end{pmatrix} \begin{pmatrix}
1 + f_{1,r}f_{3,r} & f_{1,r} & 0 \\
f_{3,r} & 1 & 0 \\
0 & 0 & 1
\end{pmatrix} \begin{pmatrix}
1 & 0 & 0 \\
f_{3,a} & 1 & f_{4} \\
0 & 0 & 1
\end{pmatrix}
,$$
where $f_j  = f_{j,a} + f_{j,r}$, as well as the relation
$$\beta(k) - r_{2,a}^{*}(\omega k) r_1(k) = r_{2,a}(\omega^2 k), \qquad k \in \R_+,$$
we can factorize $v_1^{(2)}$ for $k \in \Gamma_1^{(2)}$ as follows:
\begin{align} \label{v12factorization}
v_1^{(2)} = \begin{pmatrix}
1 - \frac{\delta_{1+}^{2}}{\delta_{1-}^{2}} \hat{r}_{1}(k)\hat{r}_{1}^{*}(k) & - \frac{\delta_{3}\delta_{5}}{\delta_{1-}^{2}} \hat{r}_{1}(k) e^{-t\Phi_{21}} & \frac{\delta_{3}^{2}}{\delta_{1-}\delta_{5} } \beta(k)e^{-t\Phi_{31}} \\
\frac{\delta_{1+}^{2}}{\delta_{3}\delta_{5}} \hat{r}_{1}^{*}(k) e^{t\Phi_{21}} & 1 & -r_2^{*}(\omega k) \frac{\delta_{1+}\delta_{3}}{\delta_{5}^{2}} e^{-t \Phi_{32}} \\
0 & 0 & 1
\end{pmatrix}
= v_1^{(2)A}v_{1,r}^{(2)}v_1^{(2)B},
\end{align}
where
\begin{align*}
& v_1^{(2)A}=
\begin{pmatrix}
1 & - \frac{\delta_{3} \delta_{5}}{\delta_{1-}^{2}} \hat{r}_{1,a}(k)e^{-t \Phi_{21}} & \frac{\delta_{3}^{2}}{\delta_{1-}\delta_{5}} r_{2,a}(\omega^{2}k)e^{-t\Phi_{31}} \\
0 & 1 & 0 \\
0 & 0 & 1
\end{pmatrix},
	\\
& v_{1,r}^{(2)} = \begin{pmatrix}
1-\frac{\delta_{1+}^{2}}{\delta_{1-}^{2}}\hat{r}_{1,r}^{*}(k) \hat{r}_{1,r}(k) & - \frac{\delta_{3}\delta_{5}}{\delta_{1-}^{2}} \hat{r}_{1,r}(k)e^{-t \Phi_{21}} & 0 \\
\frac{\delta_{1+}^{2}}{\delta_{3}\delta_{5}} \hat{r}_{1,r}^{*}(k)e^{t \Phi_{21}} & 1 & 0 \\
0 & 0 & 1
\end{pmatrix},
	\\	
& v_1^{(2)B} = \begin{pmatrix}
1 & 0 & 0 \\
\frac{\delta_{1+}^{2}}{\delta_{3}\delta_{5}} \hat{r}_{1,a}^{*}(k)e^{t \Phi_{21}} & 1 & - \frac{\delta_{1+}\delta_{3}}{\delta_{5}^{2}} r_{2,a}^{*}(\omega k)e^{-t \Phi_{32}} \\
0 & 0 & 1
\end{pmatrix}.
\end{align*}

Similarly, using the general identity
\begin{align*}
& \begin{pmatrix}
1 & f_{1} & f_{2} \\
f_{3} & 1+f_{1} f_{3} & f_{4} \\
0 & 0 & 1
\end{pmatrix} = \begin{pmatrix}
1 & 0 & 0 \\
f_{3,a} & 1 & f_{4,a}-f_{2,a}f_{3,a} \\
0 & 0 & 1
\end{pmatrix} \\
& \times \begin{pmatrix}
1 & f_{1,r} & f_{2,r} \\
f_{3,r} & 1+f_{1,r} f_{3,r} & f_{4,r} - f_{2,a} f_{3,r} - f_{2,r} f_{3,a} \\
0 & 0 & 1
\end{pmatrix} \begin{pmatrix}
1 & f_{1,a} & f_{2,a} \\
0 & 1 & 0 \\
0 & 0 & 1
\end{pmatrix},
\end{align*}
where $f_j  = f_{j,a} + f_{j,r}$, as well as the relation
$$\alpha(k) - r_1^*(k) \beta(k) = -r_{2,a}^{*}(\omega k) - r_1^*(k)r_{2,a}(\omega^2 k), \qquad k \in \R_+,$$
we can factorize $v_7^{(2)}$ for $k \in \Gamma_7^{(2)}$ as follows:
\begin{align*}
v_7^{(2)} = & \begin{pmatrix}
1 & - \frac{\delta_{3}\delta_{5}}{\delta_{1}^{2}} r_1(k)e^{-t\Phi_{21}} & \frac{\delta_{3}^{2}}{\delta_{1}\delta_{5}} \beta(k)e^{-t\Phi_{31}} \\
\frac{\delta_{1}^{2}}{\delta_{3}	\delta_{5}} r_1^{*}(k)e^{t\Phi_{21}} & 1 - r_1(k)r_1^{*}(k) & \frac{\delta_{1}\delta_{3}}{\delta_{5}^{2}} \alpha(k) e^{-t\Phi_{32}} \\
0 & 0 & 1
\end{pmatrix}
= v_7^{(2)A}v_{7,r}^{(2)}v_7^{(2)B},
\end{align*}
where
\begin{align*}
& v_7^{(2)A} = \begin{pmatrix}
1 & 0 & 0 \\
\frac{\delta_{1}^{2}}{\delta_{3}\delta_{5}}r_{1,a}^{*}(k)e^{t\Phi_{21}} & 1 & -\frac{\delta_{1}\delta_{3}}{\delta_{5}^{2}} \big( r_{2,a}^{*}(\omega k) + r_{1,a}^{*}(k)r_{2,a}(\omega^{2}k) \big) e^{-t \Phi_{32}} \\
0 & 0 & 1
\end{pmatrix},
	\\
& v_{7,r}^{(2)} = \begin{pmatrix}
1 & - \frac{\delta_{3}\delta_{5}}{\delta_{1}^{2}} r_{1,r}(k)e^{-t\Phi_{21}} & \frac{\delta_{3}^{2}}{\delta_{1}\delta_{5}} \beta_{r}(k)e^{-t\Phi_{31}} \\
\frac{\delta_{1}^{2}}{\delta_{3}\delta_{5}} r_{1,r}^{*}(k)e^{t \Phi_{21}} & 1-|r_{1,r}(k)|^2 & \frac{\delta_{1}\delta_{3}}{\delta_{5}^{2}} r_{1,r}^{*}(k) \big( r_{1,r}(k)r_{2,a}^{*}(\omega k) -r_{2,a}(\omega^{2}k)\big)e^{-t\Phi_{32}} \\
0 & 0 & 1
\end{pmatrix},
	\\
& v_7^{(2)B} = \begin{pmatrix}
1 & - \frac{\delta_{3}\delta_{5}}{\delta_{1}^{2}} r_{1,a}(k)e^{-t\Phi_{21}} & \frac{\delta_{3}^{2}}{\delta_{1}\delta_{5}} \beta_{a}(k)e^{-t\Phi_{31}} \\
0 & 1 & 0 \\
0 & 0 & 1
\end{pmatrix},
\end{align*}
and
\begin{align*}
\beta_{r}(k) := r_{1,r}(k)r_{2,a}^{*}(\omega k), \qquad \beta_{a}(k) := r_{2,a}(\omega^{2}k) + r_{1,a}(k)r_{2,a}^{*}(\omega k).
\end{align*}

\begin{figure}
\begin{center}
\bigskip\bigskip
 \begin{overpic}[width=.55\textwidth]{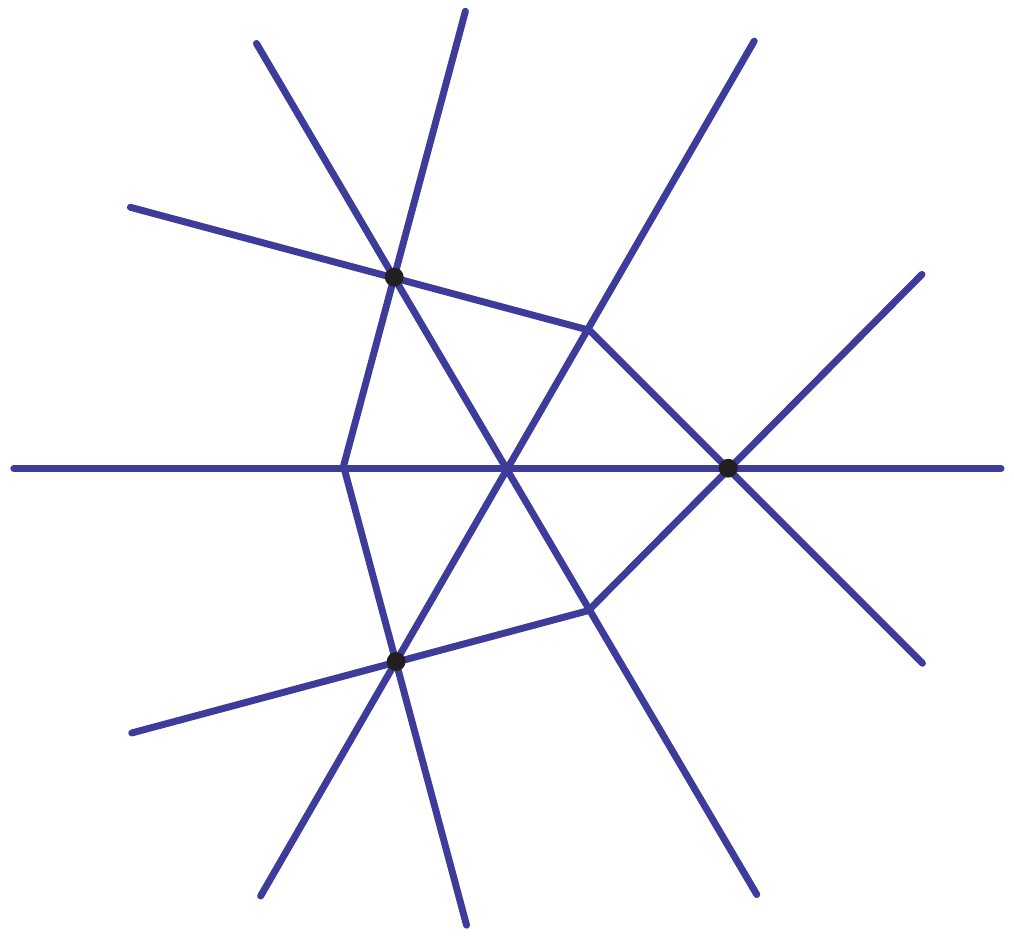}
 \put(102,42){\small $\re k$}
 \put(85,52){\small $V_1$}
 \put(57,50){\small $V_2$}
 \put(57,40){\small $V_3$}
 \put(85,38){\small $V_4$}
 \put(70.5,41){\small $k_0$}
 \put(43,65){\small $\omega k_0$}
 \put(41,22){\small $\omega^2 k_0$}
   \end{overpic}
   \begin{figuretext}\label{Vjs.pdf}
      The open sets $\{V_j\}_1^4$ in the complex $k$-plane.
      \end{figuretext}
   \end{center}
\end{figure}

Let $V_j \equiv V_j(\zeta) \subset \C$, $j = 1, \dots, 4$, denote the open subsets of the complex $k$-plane displayed in Figure \ref{Vjs.pdf}.
Define the sectionally analytic function $m^{(3)}$ by
$$m^{(3)}(x,t,k) = m^{(2)}(x,t,k)H(x,t,k),$$
where $H$ is defined for $k \in D_1 \cup D_6$ by
\begin{align}\label{Hdef}
H(x,t,k) = \begin{cases} (v_1^{(2)B})^{-1}, & k \in V_1, \\
 (v_7^{(2)B})^{-1}, & k \in V_2, \\
v_7^{(2)A}, & k \in V_3, \\
v_1^{(2)A}, & k \in V_{4}, \\
I, & \text{elsewhere in $D_1 \cup D_6$,}
\end{cases}
\end{align}
and extended to all of $\C \setminus \Gamma$ by means of the symmetry
$H(x,t,k) =  \mathcal{A} H(x,t,\omega k)\mathcal{A}^{-1}$.
Let $\Gamma^{(3)}$ be the contour displayed in Figure \ref{Gamma3.pdf}.

\begin{figure}
\begin{center}
 \begin{overpic}[width=.7\textwidth]{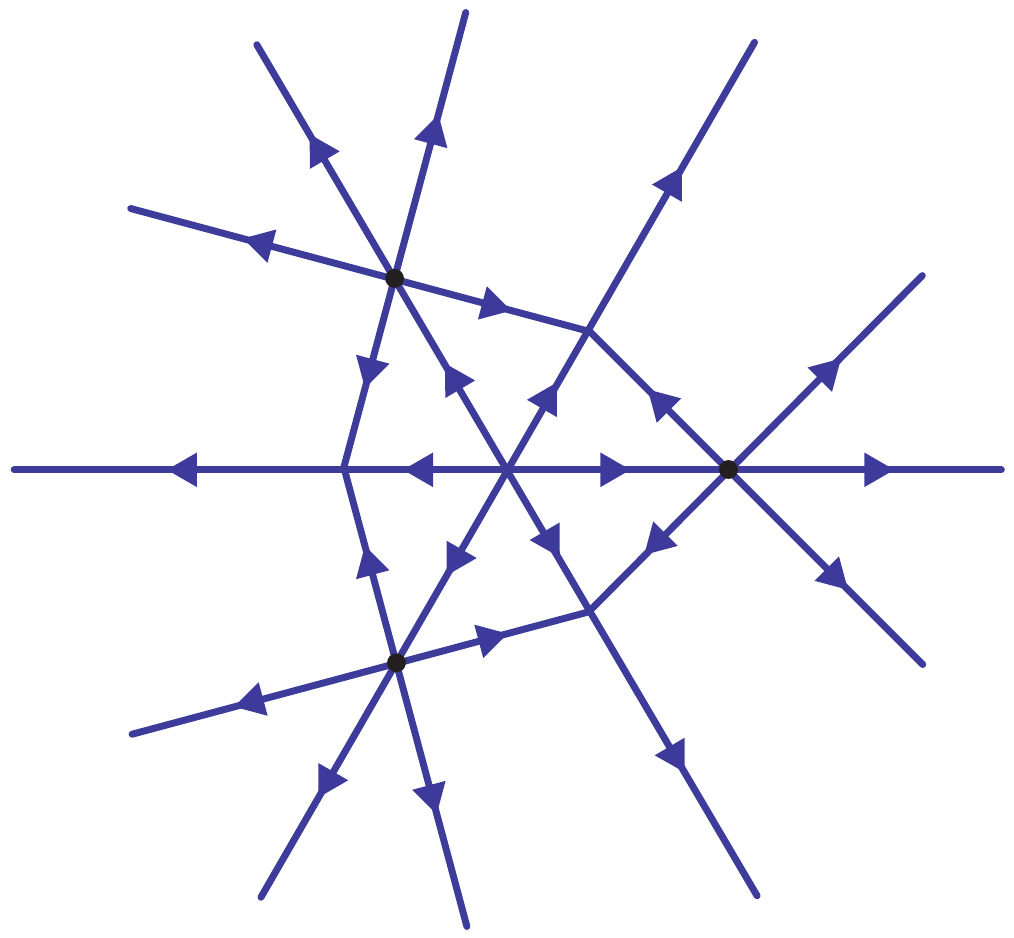}
  \put(101,45){\small $\Gamma^{(3)}$}
 \put(83.5,54){\small $1$}
 \put(65,54.5){\small $2$}
 \put(65,35){\small $3$}
 \put(84,35){\small $4$}
 \put(68,73){\small $5$}
 \put(55.5,52){\small $6$}
 \put(60,42){\small $7$}
 \put(86,42){\small $8$}
 \put(70.5,41){\small $k_0$}
 \put(43,65){\small $\omega k_0$}
 \put(41,22){\small $\omega^2 k_0$}
   \end{overpic}
     \begin{figuretext}\label{Gamma3.pdf}
       The contour $\Gamma^{(3)}$ in the complex $k$-plane.
     \end{figuretext}
     \end{center}
\end{figure}

\begin{lemma}\label{Hlemma}
$H(x,t,k)$ is uniformly bounded for $k \in \mathbb{C}\setminus \Gamma^{(3)}$, $t> 0$, and $\zeta \in \mathcal{I}$.
Moreover, $H = I + O(k^{-1})$ as $k \to \infty$.
\end{lemma}
\begin{proof}
We present the proof for $k \in V_{1} \cup V_{2}$; the proof for $k \in V_3 \cup V_4$ is similar. Note that $V_{j} \subset U_{j}$, $j=1,...,4$ (see Figures \ref{Ujs.pdf} and \ref{Vjs.pdf}). Note also the identities
\begin{align}\label{Phi relation 21 and 32} 
& \Phi_{21}(\zeta,\omega k) = \Phi_{32}(\zeta,k), 
\qquad \Phi_{21}(\zeta,\omega^{2}k) = -\Phi_{31}(\zeta,k), 
\qquad \Phi_{21} + \Phi_{32} = \Phi_{31}. 
\end{align}
If $k \in \bar{V}_{1}$, then $\omega k \in \bar{U}_{2}$ and (see Figures \ref{rePhi23.pdf} and \ref{rePhi12.pdf})
\begin{align*}
\re \Phi_{21}(\zeta, k) \leq 0, \quad \re \Phi_{32}(\zeta, k) \geq 0.
\end{align*}
Therefore, using \eqref{Phi relation 21 and 32}, \eqref{estimate r2astar at infty}, \eqref{estimate r1ahatstar at infty}, and \eqref{delta1bound}, we find
\begin{align*}
& |((v_{1}^{(2)B})^{-1})_{21}| = \bigg| \frac{\delta_{1}^{2}}{\delta_{3}\delta_{5}} \hat{r}_{1,a}^{*}(k)e^{t\Phi_{21}}\bigg| \leq \frac{C}{1+|k|}e^{-ct |\re\Phi_{21}|}, \qquad k \in V_{1}, \\
& |((v_{1}^{(2)B})^{-1})_{23}| = \bigg| \frac{\delta_{1}\delta_{3}}{\delta_{5}^{2}} r_{2,a}^{*}(\omega k) e^{-t\Phi_{32}} \bigg|\leq \frac{C}{1+|k|}e^{-ct |\re\Phi_{32}|}, \qquad k \in V_{1}.
\end{align*}
This proves the claim for $k \in V_{1}$. All entries of $(v_{7}^{(2)B})^{-1}$ are continuous functions on $\bar{V}_{2}$. Since $\bar{V}_{2}$ is compact, the claim follows also for $k \in V_{2}$.
\end{proof}
It follows from Lemma \ref{Hlemma} that $m$ satisfies RH problem \ref{RHm} if and only if $m^{(3)}$
satisfies RH problem \ref{RHmj} with $j = 3$, where  $\Gamma^{(3)}$ is the contour displayed in Figure \ref{Gamma3.pdf} and the jump matrix $v^{(3)}$ is given for $-\pi/3 < \arg k\leq \pi/3$ by
\begin{align*}
v_1^{(3)} & = v_1^{(2)B} = \begin{pmatrix}
1 & 0 & 0 \\
\frac{\delta_{1}^{2}}{\delta_{3}\delta_{5}} \hat{r}_{1,a}^{*}(k)e^{t \Phi_{21}} & 1 &- \frac{\delta_{1}\delta_{3}}{\delta_{5}^{2}} r_{2,a}^{*}(\omega k)e^{-t \Phi_{32}} \\
0 & 0 & 1
\end{pmatrix},
	\\
v_2^{(3)} & = (v_7^{(2)B})^{-1}
= \begin{pmatrix}
1 & \frac{\delta_{3}\delta_{5}}{\delta_{1}^{2}}r_{1,a}(k)e^{-t\Phi_{21}} & -\frac{\delta_{3}^{2}}{\delta_{1}\delta_{5}}\big( r_{2,a}(\omega^{2}k) + r_{1,a}(k)r_{2,a}^{*}(\omega k) \big)e^{-t\Phi_{31}} \\
0 & 1 & 0 \\
0 & 0 & 1
\end{pmatrix},
	\\
 v_3^{(3)} & = (v_7^{(2)A})^{-1}
= \begin{pmatrix}
1 & 0 & 0 \\
- \frac{\delta_{1}^{2}}{\delta_{3}\delta_{5}}r_{1,a}^{*}(k)e^{t\Phi_{21}} & 1 & \frac{\delta_{1}\delta_{3}}{\delta_{5}^{2}}\big( r_{2,a}^{*}(\omega k)+r_{1,a}^{*}(k)r_{2,a}(\omega^{2}k) \big)e^{-t\Phi_{32}} \\
0 & 0 & 1
\end{pmatrix},
	\\
 v_4^{(3)} & = v_1^{(2)A} =  \begin{pmatrix}
1 & - \frac{\delta_{3}\delta_{5}}{\delta_{1}^{2}}\hat{r}_{1,a}(k)e^{-t \Phi_{21}} & \frac{\delta_{3}^{2}}{\delta_{1}\delta_{5}}r_{2,a}(\omega^{2}k)e^{-t\Phi_{31}} \\
0 & 1 & 0 \\
0 & 0 & 1
\end{pmatrix},
	\\
 v_5^{(3)} & = v_2^{(2)}  = \begin{pmatrix}
1 & 0 & 0 \\
0 & 1-r_{2,r}(\omega k)r_{2,r}^{*}(\omega k) & - \frac{\delta_{1}\delta_{3}}{\delta_{5}^{2}}r_{2,r}^{*}(\omega k)e^{-t\Phi_{32}} \\
0 & \frac{\delta_{5}^{2}}{\delta_{1}\delta_{3}}r_{2,r}(\omega k)e^{t\Phi_{32}} & 1
\end{pmatrix},
	\\
 v_6^{(3)}& = v_7^{(2)B}v_2^{(2)} \mathcal{A}^{-1} v_7^{(2)A}(x,t,\omega^2 k) \mathcal{A}
	= \begin{pmatrix}
1 & \frac{\delta_{3}\delta_{5}}{\delta_{1}^{2}}g(k)e^{-t\Phi_{21}} & \frac{\delta_{3}^{2}}{\delta_{1}\delta_{5}}f(k) e^{-t \Phi_{31}} \\
0 & 1 - r_{2,r}(\omega k) r_{2,r}^{*}(\omega k) & -\frac{\delta_{1}\delta_{3}}{\delta_{5}^{2}}r_{2,r}^{*}(\omega k)e^{-t \Phi_{32}} \\
0 & \frac{\delta_{5}^{2}}{\delta_{1}\delta_{3}}r_{2,r}(\omega k) e^{t \Phi_{32}} & 1
\end{pmatrix},
	\\
v_7^{(3)} & =  v_{7,r}^{(2)} = \begin{pmatrix}
1 & - \frac{\delta_{3}\delta_{5}}{\delta_{1}^{2}}r_{1,r}(k)e^{-t\Phi_{21}} & \frac{\delta_{3}^{2}}{\delta_{1}\delta_{5}} r_{1,r}(k)r_{2,a}^{*}(\omega k) e^{-t\Phi_{31}} \\
\frac{\delta_{1}^{2}}{\delta_{3}\delta_{5}}r_{1,r}^{*}(k)e^{t \Phi_{21}} & 1-r_{1,r}(k)r_{1,r}^{*}(k) & \frac{\delta_{1}\delta_{3}}{\delta_{5}^{2}}r_{1,r}^{*}(k) h(k) e^{-t\Phi_{32}} \\
0 & 0 & 1
\end{pmatrix},
	\\
v_8^{(3)} & = v_{1,r}^{(2)} = \begin{pmatrix}
1- \frac{\delta_{1+}^{2}}{\delta_{1-}^{2}} \hat{r}_{1,r}(k)\hat{r}_{1,r}^{*}(k) & - \frac{\delta_{3}\delta_{5}}{\delta_{1-}^{2}}\hat{r}_{1,r}(k)e^{-t \Phi_{21}} & 0 \\
\frac{\delta_{1+}^{2}}{\delta_{3}\delta_{5}}\hat{r}_{1,r}^{*}(k)e^{t \Phi_{21}} & 1 & 0 \\
0 & 0 & 1
\end{pmatrix},
\end{align*}			
and extended to the remainder of $\Gamma^{(3)}$ by means of the first symmetry in (\ref{Z3vjsymm}).
Here the functions $f(k) \equiv f(x,t,k)$, $g(k) \equiv g(x,t,k)$, and $h(k) \equiv h(x,t,k)$ are defined by
\begin{align*}
& f(k) =  r_{2,a}(\omega^{2}k) + r_{1,a}(k)r_2^{*}(\omega k) + r_{1,a}^{*}(\omega^{2}k),
	\\
& g(k) = r_{2,r}(\omega k) \big(  r_{2,a}(\omega^{2}k)+r_{1,a}(k)r_{2,a}^{*}(\omega k) \big) - r_{1,a}(k)\big( 1-r_{2,r}(\omega k)r_{2,r}^{*}(\omega k) \big) \\
& \qquad \quad -r_{2,a}^{*}(k) - r_{1,a}^{*}(\omega^{2} k) r_{2,a}(\omega k),\\
& h(k) = r_{1,r}(k) r_{2,a}^{*}(\omega k) -r_{2,a}(\omega^2 k).
\end{align*}

The next lemma establishes bounds on $f$ and $g$ and their $x$-derivatives.
\begin{lemma}
For $k \in \Gamma_{3}$, and $l=0,1$, we have
\begin{align}\label{bound on f and g}
& |\partial_{x}^{l}f(k)| \leq C |k|e^{\frac{t}{4}|\re \Phi_{21}(\zeta, k)|},  & & |\partial_{x}^{l}g(k)| \leq C |k|e^{\frac{t}{4}|\re \Phi_{21}(\zeta, k)|}.
\end{align}
\end{lemma}
\begin{proof}
By \eqref{estimate r2astar at 0} and \eqref{estimate r1a at 0}, we have $r_{2,r}^{*}(0) = 0$, $r_{2,a}^{*}(0) = r_2^{*}(0)$, $r_{1,r}(0) = 0$, and $r_{1,a}(0) = r_1(0)$. Since $r_1(0) = \omega$ and $r_2(0) = 1$ (see \cite{CharlierLenells}), we deduce that
\begin{align*}
& r_{2,a}^{*}(0) = r_2^{*}(0) = 1 \quad \mbox{ and } \quad r_{1,a}(0) = r_1(0) = \omega.
\end{align*}
In particular,
\begin{align}\label{r2r1at0}
r_{1,a}(0)r_2^{*}(0)+r_{1,a}^{*}(0) + r_{2,a}(0) = \omega + \bar{\omega} + 1 = 0,
\end{align}

To derive the estimate for $f$, we write
\begin{align*}\nonumber
f(k)=&\; r_{1,a}^{*}(\omega^{2}k) + r_{2,a}(\omega^{2}k) + r_{1,a}(k)r_2^{*}(\omega k)
	\\
= &\; r_{1,a}^*(\omega^2 k)-r_{1,a}^{*}(0) + r_{2,a}(\omega^2 k)-r_{2,a}(0)+(r_{1,a}(k)-r_{1,a}(0))r_2^{*}(\omega k)
	\\
&+r_{1,a}(0)(r_2^{*}(\omega k)-r_2^{*}(0)) + r_{1,a}(0)r_2^{*}(0)+r_{1,a}^{*}(0) + r_{2,a}(0).
\end{align*}
Using (\ref{r2r1at0}) and the fact that $\Gamma^{(6)}_{3} \subset (\omega^{2} \mathbb{R}_{-} \cap U_{2})$, the inequalities \eqref{estimate r2astar at 0} and \eqref{estimate r1a at 0} imply 
\begin{align*}
|f(k)| \leq &\; |r_{1,a}^*(\omega^2 k)-r_{1,a}^{*}(0)| + |r_{2,a}(\omega^2 k)-r_{2,a}(0)|+C|r_{1,a}(k)-r_{1,a}(0)|+C|k| \nonumber 
	\\ \nonumber
\leq &\; C |k| (e^{\frac{t}{4}|\re \Phi_{21}(\zeta, \omega \bar{k})|}+e^{\frac{t}{4}|\re \Phi_{21}(\zeta, \omega \bar{k})|}+e^{\frac{t}{4}|\re \Phi_{21}(\zeta, k)|}) 
	\\
\leq &\; C |k|e^{\frac{t}{4}|\re \Phi_{21}(\zeta, k)|}, \qquad k \in \Gamma^{(6)}_{3}. 
\end{align*}
The estimate for $\partial_{x}f$ is derived in a similar way. Writing
\begin{align*}
g(k) = r_{2,r}(\omega k) \big(  r_{2,a}(\omega^{2}k)+r_{1,a}(k)r_{2,a}^{*}(\omega k) \big) + r_{1,a}(k) r_{2,r}(\omega k)r_{2,r}^{*}(\omega k) - f^{*}(\omega^{2}k),
\end{align*}
and using that $r_{2,r}$ and $\partial_xr_{2,r}$ vanish at $k = 0$, the estimates for $g$ and $\partial_x g$ follow from the estimates for $f$ and $\partial_x f$.
\end{proof}

\begin{lemma}\label{v3lemma}
The jump matrix $v^{(3)}$ (resp. $\partial_{x}v^{(3)}$) converges to the identity matrix $I$ (resp. to the zero matrix $0$) as $t \to \infty$ uniformly for $\zeta \in \mathcal{I}$ and $k \in \Gamma^{(3)}$ except near the three critical points $\{k_0, \omega k_0, \omega^2 k_0\}$. Moreover, the jump matrices $v_{j}^{(3)}$, $j=5,6,7,8$, satisfy
\begin{subequations}
\begin{align}\label{v3estimatesa}
& \|(1+|\cdot|) \partial_{x}^{l}(v^{(3)} - I)\|_{(L^1 \cap L^\infty)(\Gamma_5^{(3)})} \leq Ct^{-3/2}, 
	\\ \label{v3estimatesb}
& \| (1+|\cdot|)\partial_{x}^{l}(v^{(3)} - I)\|_{L^1(\Gamma_6^{(3)})} \leq Ct^{-3/2},
	\\ \label{v3estimatesc}
& \| (1+|\cdot|)\partial_{x}^{l}(v^{(3)} - I)\|_{L^\infty(\Gamma_6^{(3)})} \leq Ct^{-1}, 
	\\ \label{v3estimatesd}
& \|(1+|\cdot|) \partial_{x}^{l}(v^{(3)} - I)\|_{(L^1 \cap L^\infty)(\Gamma_7^{(3)} \cup \Gamma_8^{(3)})} \leq Ct^{-3/2},
\end{align}
\end{subequations}
uniformly for $\zeta \in \mathcal{I}$ and $l =0,1$.
\end{lemma}

\begin{proof}
Consider first the jump matrix $v_{1}^{(3)}$. Since $\re \Phi_{32} \geq c > 0$ and $\re \Phi_{21} \leq 0$ for $k\in \Gamma_1^{(3)}$,  $v_{1}^{(3)}$ (resp. $\partial_{x} v_{1}^{(3)}$) converges to $I$ (resp. to the zero matrix) as $t \to \infty$ by \eqref{Phi relation 21 and 32}, \eqref{bounds on analytic part}, and \eqref{delta1bound}. Note however that the convergence to $0$ of the (21) entry is not uniform for $k$ near $k_{0}$, because $\re \Phi_{21}(\zeta, k_{0}) = 0$. Analogous statements for $v_{2}^{(3)}, v_{3}^{(3)}$, and $v_{4}^{(3)}$ can be proved in a similar way.

Since $\re \Phi_{32} = 0$ for $k \in \Gamma_5^{(3)}$, (\ref{v3estimatesa}) follows from \eqref{norm remainder r2r}, and \eqref{delta1bound}.

We next show (\ref{v3estimatesb}) and (\ref{v3estimatesc}).
We parametrize $\Gamma_6^{(3)}$ by $ue^{\frac{\pi i}{3}}$, $0 \leq u \leq \frac{2k_0}{1 + \sqrt{3}}$, and note that
$$\re \Phi_{31}(\zeta, ue^{\frac{\pi i}{3}})=
\re \Phi_{21}(\zeta, ue^{\frac{\pi i}{3}}) =\frac{3}{2}u(2k_0-u), \qquad u \in \R.$$
It follows that
$$\begin{cases}
\re \Phi_{31}(\zeta, k) \geq \tfrac{4}{3}k_0 |k|, \\
\re \Phi_{21}(\zeta, k) \geq \tfrac{4}{3}k_0 |k|, 
\end{cases} \quad k \in \Gamma_6^{(3)}.$$
Using \eqref{bound on f and g}, \eqref{delta1bound}, \eqref{der x of delta1}, and the fact that $\partial_{x}(t\Phi_{31})=(1-\omega)k$, we thus find
\begin{align*}
& |(v_6^{(3)} - I)_{13}| \leq C |f(k)| e^{-t\re \Phi_{31}}
\leq C |k| e^{-tk_0 |k|}, & & k \in \Gamma_6^{(3)}, \\
& |\partial_{x}(v_6^{(3)})_{13}| \leq C |k| e^{-tk_0 |k|}, & & k \in \Gamma_6^{(3)}.
\end{align*}
Hence, for $l=0,1$, we have
\begin{align*}
\| (1+|\cdot|) \partial_{x}^{l}(v_6^{(3)} - I)_{13}\|_{L^1(\Gamma_6^{(3)})} \leq \frac{C}{(k_0 t)^2},
\qquad 
 \| (1+|\cdot|) \partial_{x}^{l}(v_6^{(3)} - I)_{13}\|_{L^\infty(\Gamma_6^{(3)})} \leq \frac{C}{k_0 t}, 
\end{align*}
and similar estimates apply to the $(12)$-entry.
On the other hand, $\re \Phi_{32} = 0$ for $k \in \Gamma_6^{(3)}$, and hence we can estimate the $(23)$-entry using \eqref{delta1bound} as follows:
$$|(v_6^{(3)} - I)_{23}| = |\tfrac{\delta_{1}\delta_{3}}{\delta_{5}^{2}} r_{2,r}^{*}(\omega k)|
\leq C |r_{2,r}^{*}(\omega k)|, \qquad k \in \Gamma_6^{(3)}.$$
By \eqref{norm remainder r2r}, this implies that the $L^1$ and $L^\infty$ norms of $(1+|\cdot|)(v^{(3)} - I)_{23}$ on $\Gamma_6^{(3)}$ are $O(t^{-3/2})$ as $t \to \infty$. Using also \eqref{der x of delta1} and \eqref{norm remainder r2r}, we conclude similarly that the $L^1$ and $L^\infty$ norms of $(1+|\cdot|)\partial_{x}v^{(3)}_{23}$ on $\Gamma_6^{(3)}$ are $O(t^{-3/2})$ as $t \to \infty$. A similar estimate applies to the $(32)$-entry and its $x$-derivative. The $(22)$-entry is even smaller. This proves (\ref{v3estimatesb}) and (\ref{v3estimatesc}).

We finally show (\ref{v3estimatesd}).
Note that $\re \Phi_{32} > 0$ and $\re \Phi_{31} > 0$ for $k \in \R_+$. We conclude from \eqref{estimate r2astar at infty} that $|(v_7^{(3)} - I)_{23}|$ and
 $|(v_7^{(3)} - I)_{13}|$ decay to zero as $t \to \infty$ faster than $|(v_7^{(3)} - I)_{12}|$ and $|(v_7^{(3)} - I)_{21}|$. Moreover, since $\re \Phi_{21} =  0$ for $k \in \R_+$, \eqref{norm remainder r1r} and \eqref{delta1bound} imply
$$
|(v_7^{(3)} - I)_{21}|=|\tfrac{\delta_{1}^{2}}{\delta_{3}\delta_{5}} r_{1,r}^*|\leq Ct^{-3/2},
\qquad
|(v_7^{(3)} - I)_{12}|=|\tfrac{\delta_{3}\delta_{5}}{\delta_{1}^{2}} r_{1,r}|\leq Ct^{-3/2},
$$
while $|(v_7^{(3)} - I)_{22}|$ is even smaller. Thus,
\begin{align*}
\|v^{(3)} - I\|_{(L^1 \cap L^\infty)(\Gamma_7^{(3)})} \leq Ct^{-3/2}.
\end{align*}
To estimate $\partial_{x}(v_{7}^{(3)})_{21}$, we use \eqref{norm remainder r1r} and \eqref{der x of delta1}. This gives
\begin{align*}
|\partial_{x}(v_{7}^{(3)})_{21}| \leq \Big| \partial_{x} \Big( \tfrac{\delta_{3}\delta_{5}}{\delta_{1}^{2}}\Big) r_{1,r}\Big| + \Big| \tfrac{\delta_{3}\delta_{5}}{\delta_{1}^{2}} \partial_{x}r_{1,r} \Big| \leq Ct^{-3/2}.
\end{align*}
The entries $\partial_{x}(v_{7}^{(3)})_{12}$ and $\partial_{x}(v_{7}^{(3)})_{22}$ are estimated in a similar way.

The matrix $v_{8}^{(3)}$ can be estimated in the same way as $v_{7}^{(3)}$, except that now we need to use \eqref{norm remainder r1rhat} and to note that $\re \Phi_{21}=0$ for $k \in (k_0, \infty)$. This proves (\ref{v3estimatesd}).
\end{proof}

\section{Local parametrix at $k_0$}\label{localsec}
In Section \ref{section4.5}, we arrived at a RH problem for $m^{(3)}$ with the property that the matrix $v^{(3)} - I$ decays to zero as $t \to \infty$ everywhere except near the three critical points $\{k_0, \omega k_0, \omega^2 k_0\}$. This means that we only have to consider neighborhoods of these three points when computing the long-time asymptotics of $m^{(3)}$. In this section, we find a local solution $m^{k_0}$ which approximates $m^{(3)}$ near $k_0$. The basic idea is that in the large $t$ limit, the RH problem for $m^{(3)}$ near $k_0$ reduces to an RH problem on a cross which can be solved exactly in terms of parabolic cylinder functions \cite{I1981, DZ1993}.

Let $\epsilon \equiv \epsilon(\zeta) = k_0/2$.
Let $D_\epsilon(k_0)$ denote an open disk of radius $\epsilon$ centered at $k_0$.
Let $\mathcal{D} = D_\epsilon(k_0) \cup \omega D_\epsilon(k_0) \cup \omega^2 D_\epsilon(k_0)$.
Let $\mathcal{X} = k_0 + X$, where $X$ is the contour defined in (\ref{Xdef}). We will also use the notations $\mathcal{X}^\epsilon = \mathcal{X} \cap D_\epsilon(k_0)$ and $\mathcal{X}_{j}^\epsilon = (k_{0}+X_{j}) \cap D_\epsilon(k_0)$, $j=1,...,4$, where $X_{j}$ is defined in (\ref{Xdef}).

In order to relate $m^{(3)}$ to the solution $m^X$ of Lemma \ref{Xlemma}, we make a local change of variables for $k$ near $k_0$ and introduce the new variable $z \equiv z(\zeta, k)$ by
\begin{align}\label{def of conformal map}
z = 3^{1/4}\sqrt{2t}(k-k_0).
\end{align}
For each $\zeta \in \mathcal{I}$, the map $k \mapsto z$ is a biholomorphism from $D_\epsilon(k_0)$ onto the open disk of radius $3^{1/4}\sqrt{2t} \epsilon$ centered at the origin.
Using that
$$\Phi_{21}(\zeta, k) = \Phi_{21}(\zeta, k_0) + i\sqrt{3}(k - k_0)^2,$$
where $\Phi_{21}(\zeta, k_0) = -i\sqrt{3}k_0^2$, we see that
\begin{align*}
t(\Phi_{21}(\zeta, k) - \Phi_{21}(\zeta, k_0)) = \frac{i}{2}z^2.
\end{align*}
Equations \eqref{delta1 expression in terms of log and chi1} and \eqref{def of conformal map} imply that, for $\zeta \in \mathcal{I}$ and $k \in D_{\epsilon}(k_{0}) \setminus [k_0,\infty)$, 
\begin{align*}
\frac{\delta_{3}\delta_{5}}{\delta_{1}^{2}} = e^{2i\nu \ln_{0}(z)}(2\sqrt{3}t)^{-i\nu}e^{2\chi_{1}(\zeta,k)}\delta_{3}\delta_{5} = e^{2i\nu \ln_{0}(z)}d_0(\zeta, t) d_1(\zeta,k),
\end{align*}
where the functions $d_0(\zeta, t)$ and $d_1(\zeta, k)$ are defined for $\zeta \in \mathcal{I}$ and $k \in D_{\epsilon}(k_{0}) \setminus [k_0,\infty)$ by
\begin{align}
d_0(\zeta, t) = &\; (2\sqrt{3}t)^{-i\nu} e^{2\chi_{1}(\zeta,k_{0})} \delta_3(\zeta,k_{0})\delta_5(\zeta,k_{0}), \label{d0def} \\ 
d_1(\zeta, k) = &\; e^{2\chi_{1}(\zeta,k)-2\chi_{1}(\zeta,k_{0})}\frac{\delta_3(\zeta,k)\delta_5(\zeta,k)}{\delta_3(\zeta,k_{0})\delta_5(\zeta,k_{0})}. \label{d1def}
\end{align}

Defining $\tilde{m}$ for $k$ near $k_0$ by
$$\tilde{m}(x,t,k) = m^{(3)}(x,t,k)Y(\zeta,t), \qquad k \in D_\epsilon(k_0),
$$
where
$$Y(\zeta,t) = \begin{pmatrix} 
d_0^{1/2}(\zeta, t)e^{-\frac{t}{2}\Phi_{21}(\zeta, k_0)} & 0 & 0 \\
0 & d_0^{-1/2}(\zeta, t)e^{\frac{t}{2}\Phi_{21}(\zeta, k_0)} & 0 \\
0 & 0 & 1 \end{pmatrix},$$
we find that the jump $\tilde{v}(x,t,k)$ of $\tilde{m}$ across $\mathcal{X}^\epsilon$ is given by 
\begin{align*}
& \tilde{v}_1 =
 \begin{pmatrix}
  1 & 0 & 0 \\
  e^{-2i\nu\ln_0(z)}  d_1^{-1} \hat{r}_{1,a}^{*}(k) e^{\frac{iz^2}{2}} & 1 & -\frac{\delta_{1}\delta_{3}}{\delta_{5}^{2}} d_0^{1/2} r_{2,a}^{*}(\omega k) e^{-t\Phi_{32}}e^{-\frac{t}{2}\Phi_{21}(\zeta, k_0)}
 	\\
0 & 0& 1 \end{pmatrix}, 
  	\\
& \tilde{v}_2 = \begin{pmatrix}
  1 & e^{2i\nu\ln_0(z)}  d_1 r_{1,a}(k) e^{-\frac{iz^2}{2}} & -\frac{\delta_{3}^{2}}{\delta_{1}\delta_{5}} d_0^{-1/2} \Omega_1(k) e^{-t\Phi_{31}}e^{\frac{t}{2}\Phi_{21}(\zeta, k_0)} \\
 0  & 1 & 0 \\
0 & 0 & 1 \end{pmatrix}, 
	\\
& \tilde{v}_3 = \begin{pmatrix}
  1 & 0 & 0 \\
 -e^{-2i\nu\ln_0(z)}  d_1^{-1} r_{1,a}^*(k) e^{\frac{iz^2}{2}} & 1 & \frac{\delta_{1}\delta_{3}}{\delta_{5}^{2}}d_0^{1/2} \Omega_2(k) e^{-t\Phi_{32}}e^{-\frac{t}{2}\Phi_{21}(\zeta, k_0)} \\
0 & 0 & 1  \end{pmatrix}, 
	\\
& \tilde{v}_4 = \begin{pmatrix}
  1 & -e^{2i\nu\ln_0(z)}  d_1 \hat{r}_{1,a}(k) e^{-\frac{iz^2}{2}} & \frac{\delta_{3}^{2}}{\delta_{1}\delta_{5}}d_0^{-1/2} r_{2,a}(\omega^2 k) e^{-t\Phi_{31}}e^{\frac{t}{2}\Phi_{21}(\zeta, k_0)} \\
0 & 1 & 0 \\
0 & 0 & 1 \end{pmatrix}, 
\end{align*}
where $\tilde{v}_j$ denotes the restriction of $\tilde{v}$ to $\mathcal{X}_j^\epsilon$, $j = 1,2,3,4$, and $\Omega_1(k) \equiv \Omega_1(x,t,k)$ and $\Omega_2(k) \equiv \Omega_2(x,t,k)$ are given by
\begin{align*}
& \Omega_1(k) =r_{2,a}(\omega^2 k) + r_{1,a}(k)r_{2,a} ^{*}(\omega k), 
\qquad \Omega_2(k) = r_{2,a}^{*}(\omega k) + r_{1,a}^*(k) r_{2,a}(\omega^2 k).
\end{align*}

Define $q \equiv q(\zeta)$ by
$$q = r_1(k_0).$$
For a fixed $z$, $r_{1,a}(k) \to q$,  $\hat{r}_{1,a}^{*}(k) \to \frac{\bar{q}}{1 - |q|^2}$, and $d_1(\zeta, k) \to 1$ as $t\to \infty$. This suggests that $\tilde{v}(x,t,k)$ tends to the jump matrix $v^{X}(x,t,z)$ defined in (\ref{vXdef}) for large $t$.
In other words, that the jumps of $m^{(3)}$ for $k$ near $k_0$ approach those of the function $m^X Y^{-1}$  as $t \to \infty$.
This suggests that we approximate $m^{(3)}$ in the neighborhood $D_\epsilon(k_0)$ of $k_0$ by the $3 \times 3$-matrix valued function $m^{k_0}$ defined by
\begin{align}\label{mk0def}
m^{k_0}(x,t,k) = Y(\zeta,t) m^X(q(\zeta),z(\zeta, k)) Y(\zeta,t)^{-1}, \qquad k \in D_{\epsilon}(k_{0}).
\end{align}
The prefactor $Y(\zeta,t)$ on the right-hand side of (\ref{mk0def}) is included so that $m^{k_0} \to I$ on $\partial D_\epsilon(k_0)$ as $t \to \infty$; this ensures that $m^{k_0}$ is a good approximation of $m^{(3)}$ in $D_\epsilon(k_0)$ for large $t$.


\begin{lemma}\label{lemma: bound on Y}
The function $Y(\zeta,t)$ is uniformly bounded:
\begin{align}\label{Ybound}
\sup_{\zeta \in \mathcal{I}} \sup_{t \geq 2} |\partial_{x}^{l}Y(\zeta,t)^{\pm 1}| < C, \qquad l = 0,1.
\end{align}
Moreover, the functions $d_0(\zeta, t)$ and $d_1(\zeta, k)$ satisfy
\begin{subequations}
\begin{align}\label{d0estimate}
& |d_0(\zeta, t)| = e^{2\pi \nu}, \qquad \zeta \in \mathcal{I}, \ t \geq 2, \\
& |\partial_{x}d_0(\zeta, t)| \leq C\frac{\ln t}{t}, \qquad \zeta \in \mathcal{I}, \ t \geq 2, \label{d0estimate derivative}
\end{align}
\end{subequations}
and
\begin{subequations}
\begin{align}\label{d1estimate}
& |d_1(\zeta, k) - 1| \leq C |k - k_0| (1+ |\ln|k-k_0||), & & \zeta \in \mathcal{I}, \ k \in \mathcal{X}^{\epsilon}, \\
& |\partial_{x} d_{1}(\zeta, k)| \leq \frac{C}{t} | \ln |k-k_{0}||,  & & \zeta \in \mathcal{I}, \ k \in \mathcal{X}^{\epsilon}. \label{d1estimate derivative}
\end{align}
\end{subequations}
\end{lemma}
\begin{proof}
The symmetry \eqref{symmetry of delta1} implies 
\begin{align*}
|\delta_3(\zeta,k_{0})\delta_5(\zeta,k_{0})| = |\delta_1(\zeta,\omega^{2}k_{0})\delta_1(\zeta,\omega k_{0})| = 1,
\end{align*}
and hence \eqref{d0estimate} follows because
\begin{align*}
\re \chi_{1}(\zeta,k_{0}) = \frac{1}{2\pi} \int_{k_{0}}^{\infty} \pi d\ln\big( 1-|r_1(s)|^{2} \big) = -\frac{1}{2}\ln\big( 1-|r_1(k_{0})|^{2} \big) = \pi \nu.
\end{align*}
Using \eqref{d0estimate}, we obtain
\begin{align*}
|\partial_{x} d_0(\zeta, t)| & = |d_0(\zeta, t) \partial_{x}\ln d_0(\zeta, t)|= e^{2\pi \nu} |\partial_{x}\ln d_0(\zeta, t)|  \\
& \leq C \big( |\ln t \; \partial_{x}\nu| + |\partial_{x}\chi_{1}(\zeta,k_{0})| + |\partial_{x} \ln(\delta_3(\zeta,k_{0})\delta_5(\zeta,k_{0}))| \big)
\end{align*}
and thus \eqref{d0estimate derivative} follows from \eqref{estimate der of chi} and the fact that $\partial_{x} = \frac{1}{t}\partial_{\zeta}$. Observing that $\delta_{3}$ and $\delta_{5}$ are analytic for $k \in \mathcal{X}^{\epsilon}$, \eqref{d1estimate} follows from \eqref{asymp chi1 at k0}. Finally, we have
\begin{align*}
\partial_{x}d_{1}(\zeta,k) = d_{1}(\zeta,k) \partial_{x} \log d_{1}(\zeta,k).
\end{align*}
Since $\partial_{\zeta} \log \frac{\delta_3(\zeta,k)\delta_5(\zeta,k)}{\delta_3(\zeta,k_{0})\delta_5(\zeta,k_{0})}$ is analytic and $|d_{1}(\zeta,k)| \leq C$ for $k \in \mathcal{X}^{\epsilon}$, it follows that
\begin{align*}
|\partial_{x}d_{1}(\zeta,k)| \leq C \bigg( |\partial_{x}(\chi_{1}(\zeta,k)-\chi_{1}(\zeta,k_{0}))| + \frac{1}{t}\bigg|\partial_{\zeta} \log \frac{\delta_3(\zeta,k)\delta_5(\zeta,k)}{\delta_3(\zeta,k_{0})\delta_5(\zeta,k_{0})}\bigg| \bigg),
\end{align*}
and so \eqref{d1estimate derivative} follows from \eqref{asymp chi1 der at k0}.
\end{proof}

\begin{lemma}\label{k0lemma}
 For each $(x,t)$, the function $m^{k_0}(x,t,k)$ defined in (\ref{mk0def}) is an analytic and bounded function of $k \in D_\epsilon(k_0) \setminus \mathcal{X}^\epsilon$.
Across $\mathcal{X}^\epsilon$, $m^{k_0}$ obeys the jump condition $m_+^{k_0} =  m_-^{k_0} v^{k_0}$, where the jump matrix $v^{k_0}$ satisfies
\begin{align}\label{v3vk0estimate}
\begin{cases}
 \|\partial_{x}^{l}(v^{(3)} - v^{k_0})\|_{L^1(\mathcal{X}^\epsilon)} \leq C t^{-1} \ln t,
	\\
\|\partial_{x}^{l}(v^{(3)} - v^{k_0})\|_{L^\infty(\mathcal{X}^\epsilon)} \leq C t^{-1/2} \ln t,
\end{cases} \qquad \zeta \in \mathcal{I}, \ t \geq 2, \ l=0,1.
\end{align}
Furthermore, as $t \to \infty$,
\begin{align}\label{mmodmuestimate2}
& \|\partial_{x}^{l}(m^{k_0}(x,t,\cdot)^{-1} - I)\|_{L^\infty(\partial D_\epsilon(k_0))} = O(t^{-1/2}), \qquad l=0,1,
	\\ \label{mmodmuestimate1}
& \frac{1}{2\pi i}\int_{\partial D_\epsilon(k_0)}(m^{k_0}(x,t,k)^{-1} - I) dk
= -\frac{Y(\zeta,t) m_1^{X}(q(\zeta)) Y(\zeta,t)^{-1}}{3^{1/4}\sqrt{2}\sqrt{t}} + O(t^{-1}),
\end{align}
uniformly for $\zeta \in \mathcal{I}$, and \eqref{mmodmuestimate1} can be differentiated with respect to $x$ without increasing the error term.
\end{lemma}
\begin{proof}
We have
$$v^{(3)} - v^{k_0} = Y(\zeta, t) (\tilde{v} - v^X)Y(\zeta,t)^{-1}.
$$
Thus, recalling (\ref{Ybound}), the bounds (\ref{v3vk0estimate}) follow if we can show that
\begin{subequations}\label{v3k0}
\begin{align}\label{v3k0L1}
  &\|\partial_{x}^{l}[\tilde{v}(x,t,\cdot) - v^X(x,t,z(\zeta,\cdot))]\|_{L^1(\mathcal{X}_j^\epsilon)} \leq C t^{-1} \ln t,
	\\ \label{v3k0Linfty}
&  \|\partial_{x}^{l}[\tilde{v}(x,t,\cdot) - v^X(x,t,z(\zeta,\cdot))]\|_{L^\infty(\mathcal{X}_j^\epsilon)} \leq C t^{-1/2} \ln t,
\end{align}
\end{subequations}
for $j = 1, \dots, 4$ and $l=0,1$. We give the proof of (\ref{v3k0}) for $j = 1$; similar arguments apply when $j = 2,3,4$.

For $k \in \mathcal{X}_1^\epsilon$, only the $(21)$ and $(23)$ elements of the matrix $\tilde{v} - v^X$ are nonzero.
Using \eqref{estimate r2astar at infty}, \eqref{delta1bound}, \eqref{d0estimate}, and the facts that $\Phi_{21}(\zeta, \omega k) = \Phi_{32}(\zeta, k)$ and $v_{23}^{X}(q(\zeta),z(\zeta,k)) = 0$ for $k \in \mathcal{X}_{1}^{\epsilon}$, $|(\tilde{v} - v^X)_{23}|$ can be estimated as follows:
\begin{align*}
|(\tilde{v} - v^X)_{23}| = &\; |\tfrac{\delta_{1}\delta_{3}}{\delta_{5}^{2}} d_0^{1/2} r_{2,a}^{*}(\omega k) e^{-t\Phi_{32}}e^{-\frac{t}{2}\Phi_{21}(\zeta, k_0)}|
\leq |r_{2,a}^{*}(\omega k)| e^{-t\re \Phi_{32}}\\
\leq &\;  C e^{\frac{t}{4}|\re \Phi_{21}(\zeta, \omega k)|}e^{-t\re \Phi_{32}}
= C e^{-\frac{3t}{4}|\re \Phi_{32}(\zeta, k)|}, \qquad k \in \mathcal{X}_1^{\epsilon}.
\end{align*}
For $k = k_0 + ue^{\frac{\pi i}{4}}$ and $u \geq 0$, we have
\begin{align}\label{rePhi12}
\re \Phi_{32}(\zeta, k_0 + ue^{\frac{\pi i}{4}}) = \frac{1}{2} \left(9 k_{0}^2+6 \sqrt{2} k_{0} u+\sqrt{3} u^2\right) \geq c(k_0 + u)^2.
\end{align}
Hence
$$\|(\tilde{v} - v^X)_{23}\|_{L^1(\mathcal{X}_1^\epsilon)}
\leq C \int_0^{\frac{k_{0}}{2}} e^{-ct(k_0 + u)^2} du
= C \int_{k_0}^{\frac{3k_{0}}{2}} e^{-ctv^2} dv \leq C e^{-ctk_{0}^{2}}$$
and
$$\|(\tilde{v} - v^X)_{23}\|_{L^{\infty}(\mathcal{X}_1^\epsilon)}
\leq C \sup_{u \geq 0} e^{-ct(k_0 + u)^2} \leq C e^{-ctk_{0}^{2}}.
$$

To estimate $\partial_{x}(\tilde{v} - v^X)_{23}$, we first note that
\begin{align*}
& \partial_{x}(\tilde{v} - v^X)_{23} = a_{1}+a_{2}+a_{3}+a_{4}+a_{5},
\end{align*}
where
\begin{align*}
& a_{1} = -\partial_{x}\big(\tfrac{\delta_{1}\delta_{3}}{\delta_{5}^{2}}\big) d_0^{1/2} r_{2,a}^{*}(\omega k) e^{-t\Phi_{32}}e^{-\frac{t}{2}\Phi_{21}(\zeta, k_0)}, \\
& a_{2} = -\tfrac{\delta_{1}\delta_{3}}{\delta_{5}^{2}}\partial_{x}\big( d_0^{1/2}\big) r_{2,a}^{*}(\omega k) e^{-t\Phi_{32}}e^{-\frac{t}{2}\Phi_{21}(\zeta, k_0)}, \\
& a_{3} = - \tfrac{\delta_{1}\delta_{3}}{\delta_{5}^{2}} d_0^{1/2} \partial_{x}\big(r_{2,a}^{*}(\omega k)\big) e^{-t\Phi_{32}}e^{-\frac{t}{2}\Phi_{21}(\zeta, k_0)}, \\
& a_{4} = - \tfrac{\delta_{1}\delta_{3}}{\delta_{5}^{2}} d_0^{1/2} r_{2,a}^{*}(\omega k) \partial_{x}\big(e^{-t\Phi_{32}}\big) e^{-\frac{t}{2}\Phi_{21}(\zeta, k_0)}, \\
& a_{5} = - \tfrac{\delta_{1}\delta_{3}}{\delta_{5}^{2}} d_0^{1/2} r_{2,a}^{*}(\omega k) e^{-t\Phi_{32}} \partial_{x}\big(e^{-\frac{t}{2}\Phi_{21}(\zeta, k_0)}\big).
\end{align*}
We claim that $\|a_{j}\|_{(L^{1}\cap L^{\infty})(\mathcal{X}_1^\epsilon)} \leq Ce^{-ctk_{0}^{2}}$ for $j=1,...,5$. These bounds follow from arguments which are similar to those given for $(\tilde{v} - v^X)_{23}$, but more estimates are required. For $a_{1}$, we note that $\partial_{x}\big(\tfrac{\delta_{1}\delta_{3}}{\delta_{5}^{2}}\big)$ has a pole at $k=k_{0}$ (see \eqref{der x of delta1}) which is cancelled by the zero of $r_{2,a}^{*}(\omega k)$ (see \eqref{estimate r2astar at infty}). For $a_{2}$ and $a_{3}$, we use \eqref{d0estimate derivative} and \eqref{estimate r2astar at 0}, respectively. For $a_{4}$, we note that $\partial_{x}(t\Phi_{32}) = \partial_{\zeta}(\Phi_{32}) = (1-\omega)k$, and for $a_{5}$, we observe that $\partial_{x}(t \Phi_{21}(\zeta, k_{0})) = \frac{1}{2}\partial_{k_{0}}\Phi_{21}(\zeta, k_{0}) = \omega(\omega-1)k_{0}$.
Therefore, we arrive at
\begin{align*}
\|\partial_{x}(\tilde{v} - v^X)_{23}\|_{(L^{1}\cap L^{\infty})(\mathcal{X}_1^\epsilon)} \leq Ce^{-ctk_{0}^{2}}. 
\end{align*}

We next consider the $(21)$-entry of $\tilde{v} - v^X$.
Since $q=r_1(k_0)$, from \eqref{estimate r1ahatstar at k0} it follows $\hat{r}_{1,a}^{*}(k_0) = \hat{r}_1^{*}(k_0) =\frac{\bar{q}}{1-|q|^2}$.
Furthermore, 
$$
e^{\frac{t}{4}|\re \Phi_{21}(\zeta, k)|}=e^{\frac{t}{4}|\re (\Phi_{21}(\zeta, k)-\Phi_{21}(\zeta, k_0))|}=e^{\frac{1}{4}|\re (\frac{iz^2}{2})|} \leq e^{\frac{|z|^2}{8}}.
$$
Thus $|(\tilde{v} - v^X)_{21}|$ can be estimated as follows:
\begin{align*}
|(\tilde{v} - v^X)_{21}| = &\; \bigg|e^{-2i\nu\ln_0(z)}  d_1^{-1} \hat{r}_{1,a}^{*}(k) e^{\frac{iz^2}{2}}
- \hat{r}_{1,a}^{*}(k_{0}) e^{-2i\nu\ln_0(z)} e^{\frac{iz^2}{2}}\bigg|
	\\
=& \;  \Big| e^{-2i\nu\ln_0(z)} \Big| \Big|  (d_1^{-1}-1) \hat{r}_{1,a}^{*}(k)+\big(\hat{r}_{1,a}^{*}(k)
- \hat{r}_{1,a}^{*}(k_{0})\big)\Big| \big| e^{\frac{iz^2}{2}} \big|
	\\
\leq &\; C \bigg(|d_1^{-1} - 1||\hat{r}_{1,a}^{*}(k)| + |\hat{r}_{1,a}^{*}(k)
- \hat{r}_1^{*}(k_0) | \bigg)e^{-\frac{|z|^2}{2}}.
	\\
\leq &\; C \bigg(|d_1^{-1} - 1| + |k-k_0| \bigg) e^{\frac{t}{4}|\re \Phi_{21}(\zeta, k)|} e^{-\frac{|z|^2}{2}}.
	\\
\leq &\; C \bigg(|d_1^{-1} - 1| + |k-k_0| \bigg)e^{-c t |k-k_0|^2}, \qquad k \in \mathcal{X}_1^\epsilon,
\end{align*}
where we have used \eqref{estimate r1ahatstar at k0} and \eqref{estimate r1ahatstar at infty}. Utilizing (\ref{d1estimate}), this gives
\begin{align*}
|(\tilde{v} - v^X)_{21}| \leq C |k-k_0|(1+ |\ln|k-k_0||)e^{-c t |k-k_0|^2}, \qquad k \in \mathcal{X}_1^\epsilon.
\end{align*}
Hence
$$\|(\tilde{v} - v^X)_{21}\|_{L^1(\mathcal{X}_1^\epsilon)}
\leq C \int_0^\infty u(1+ |\ln u|) e^{-ctu^2} du \leq Ct^{-1}\ln t$$
and
$$\|(\tilde{v} - v^X)_{21}\|_{L^{\infty}(\mathcal{X}_1^\epsilon)}
\leq C \sup_{u \geq 0} u(1+ |\ln u|) e^{-ctu^2} \leq Ct^{-1/2}\ln t.$$

To analyze $\partial_{x}(\tilde{v} - v^X)_{21}$, we split it into three parts as follows:
\begin{align*}
& \partial_{x}(\tilde{v} - v^X)_{21} = b_{1} + b_{2} + b_{3}, 
\end{align*}
where
\begin{align*}
& b_{1} = \partial_{x}( e^{-2i\nu\ln_0(z)}) \Big(  (d_1^{-1}-1) \hat{r}_{1,a}^{*}(k)+\big(\hat{r}_{1,a}^{*}(k)
- \hat{r}_{1,a}^{*}(k_{0})\big)\Big)  e^{\frac{iz^2}{2}}, \\
& b_{2} =  e^{-2i\nu\ln_0(z)} \partial_{x}\Big(  (d_1^{-1}-1) \hat{r}_{1,a}^{*}(k)+\big(\hat{r}_{1,a}^{*}(k)
- \hat{r}_{1,a}^{*}(k_{0})\big)\Big)  e^{\frac{iz^2}{2}}, \\
& b_{3} =  e^{-2i\nu\ln_0(z)} \Big(  (d_1^{-1}-1) \hat{r}_{1,a}^{*}(k)+\big(\hat{r}_{1,a}^{*}(k)
- \hat{r}_{1,a}^{*}(k_{0})\big)\Big) \partial_{x}e^{\frac{iz^2}{2}}.
\end{align*}
For $b_{1}$, we use that $| \partial_{x}( e^{-2i\nu\ln_0(z)}) | \leq \frac{C}{t(k-k_{0})}$ for $k \in \mathcal{X}_{1}^{\epsilon}$, and thus, by \eqref{bounds on analytic part}, 
\begin{align*}
& \|b_{1}\|_{L^1(\mathcal{X}_1^\epsilon)}
\leq C \, t^{-1} \int_0^\infty (1+ \ln u) e^{-ctu^2} du \leq Ct^{-3/2}\ln t, \\
& \|b_{1}\|_{L^{\infty}(\mathcal{X}_1^\epsilon)}
\leq C \, t^{-1} \sup_{u \geq 0} (1+ \ln u) e^{-ctu^2} \leq Ct^{-1}\ln t.
\end{align*}
The norms of $b_{2}$ and $b_{3}$ are estimated in a similar way. This completes the proof of (\ref{v3vk0estimate}).

The variable $z$ goes to infinity as $t \to \infty$ if $k \in \partial D_\epsilon(k_0)$, because
$$|z| = 3^{1/4}\sqrt{2t}|k-k_0|.$$
Thus equation (\ref{mXasymptotics}) yields
\begin{align*}
 & m^X(q(\zeta), z(\zeta, k)) = I + \frac{m_1^X(q(\zeta))}{3^{1/4}\sqrt{2t}(k-k_0)} + O(t^{-1}),
\qquad  t \to \infty,
 \end{align*}
uniformly with respect to $k \in \partial D_\epsilon(k_0)$ and $\zeta \in \mathcal{I}$, and this asymptotic formula can be differentiated with respect to $x$ without increasing the error term.
Recalling the definition (\ref{mk0def}) of $m^{k_0}$, this gives
\begin{align}\label{mmodmmuI}
 & (m^{k_0})^{-1} - I = - \frac{Y(\zeta, t)m_1^X(q(\zeta)) Y(\zeta, t)^{-1}}{3^{1/4}\sqrt{2t}(k-k_0)} + O(t^{-1}), \qquad  t \to \infty,
 \end{align}
uniformly for $k \in \partial D_\epsilon(k_0)$ and $\zeta \in \mathcal{I}$. In view of (\ref{Ybound}), the asymptotics \eqref{mmodmmuI} can be differentiated with respect to $x$. This proves (\ref{mmodmuestimate2}).
Equation (\ref{mmodmuestimate1}) follows from (\ref{mmodmmuI}) and Cauchy's formula.
\end{proof}

\section{A small-norm RH problem}\label{smallnormsec}
We use the symmetry
$$m^{k_0}(x,t,k) = \mathcal{A} m^{k_0}(x,t,\omega k)\mathcal{A}^{-1}$$
to extend the domain of definition of $m^{k_0}$ from $D_\epsilon(k_0)$ to $\mathcal{D}$, where we recall that $\mathcal{D} = D_\epsilon(k_0) \cup \omega D_\epsilon(k_0) \cup \omega^2 D_\epsilon(k_0)$.
We will show that the solution $\hat{m}(x,t,k)$ defined by
$$\hat{m} =
\begin{cases}
m^{(3)} (m^{k_0})^{-1}, & k \in \mathcal{D}, \\
m^{(3)}, & \text{elsewhere},
\end{cases}
$$
is small for large $t$. 
Let $\hat{\Gamma} = \Gamma^{(3)} \cup \partial \mathcal{D}$ be the contour displayed in Figure \ref{Gammahat.pdf} and define the jump matrix $\hat{v}$ by
$$\hat{v}= \begin{cases}
v^{(3)}, & k \in \hat{\Gamma} \setminus \bar{\mathcal{D}},
	\\
(m^{k_0})^{-1}, & k \in \partial \mathcal{D},
	\\
m_-^{k_0} v^{(3)}(m_+^{k_0})^{-1}, & k \in \hat{\Gamma} \cap \mathcal{D}.
\end{cases}
$$
The function $\hat{m}$ satisfies the following RH problem.

\begin{RHproblem}[RH problem for $\hat{m}$]\label{RHmhat}
Find a $3 \times 3$-matrix valued function $\hat{m}(x,t,\cdot) \in I + \dot{E}^{3}( \C \setminus \hat{\Gamma})$ such that $\hat{m}_+(x,t,k) = \hat{m}_-(x, t, k) \hat{v}(x, t, k)$ for a.e. $k \in \hat{\Gamma}$.
\end{RHproblem}

\begin{figure}
\begin{center}
 \begin{overpic}[width=.7\textwidth]{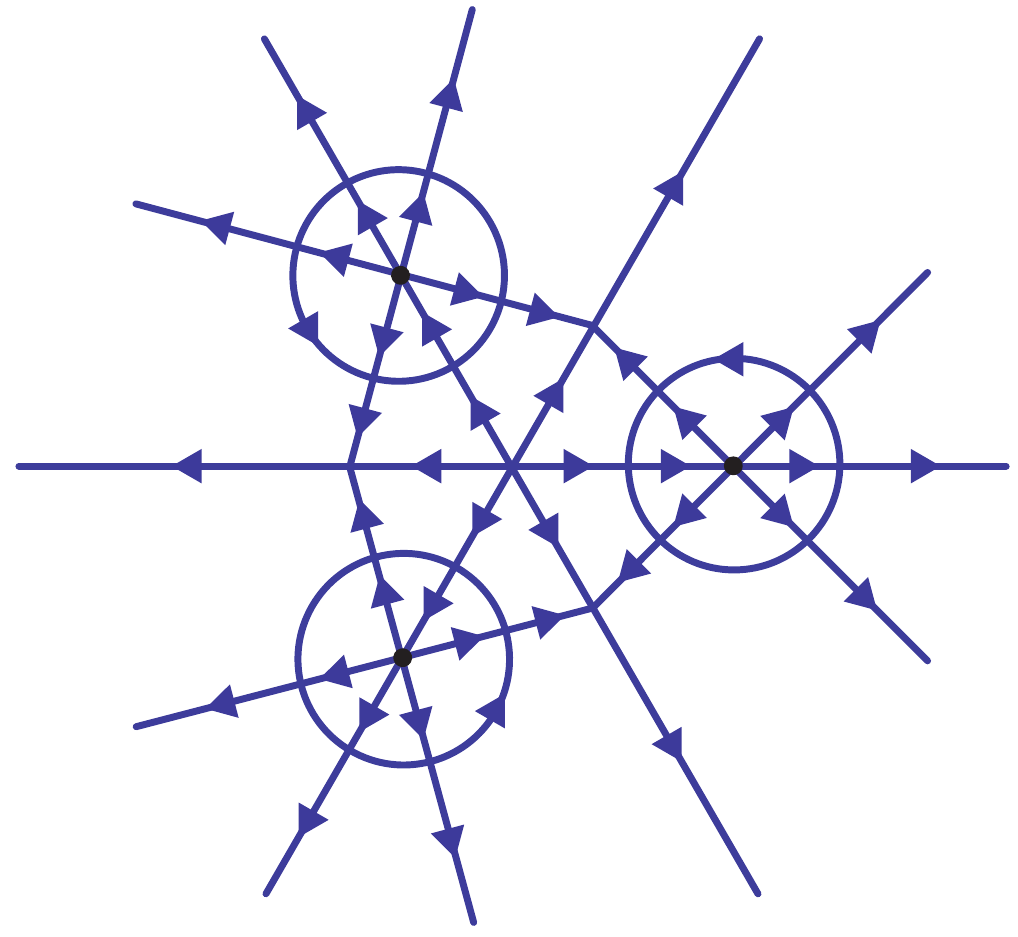}
  \put(101,45){\small $\hat{\Gamma}$}
  \put(86.5,56){\small $1$}
 \put(64,33){\small $3$}
 \put(86.5,34){\small $4$}
 \put(68,72){\small $5$}
 \put(56,52){\small $6$}
 \put(56.5,42){\small $7$}
 \put(90,42){\small $8$}
 \put(70.8,41.5){\small $k_0$}
 \put(41,66){\small $\omega k_0$}
 \put(30.5,28.5){\small $\omega^2 k_0$}
   \end{overpic}
     \begin{figuretext}\label{Gammahat.pdf}
       The contour $\hat{\Gamma} = \Gamma^{(3)} \cup \partial \mathcal{D}$ in the complex $k$-plane.
     \end{figuretext}
     \end{center}
\end{figure}

Let $\hat{\mathcal{X}}^\epsilon$ denote the union of the cross $\mathcal{X}^\epsilon$ and its images under the maps $k \mapsto\omega k$ and $k \mapsto \omega^2 k$, i.e., $\hat{\mathcal{X}}^\epsilon = \mathcal{X}^\epsilon \cup \omega \mathcal{X}^\epsilon \cup \omega^2 \mathcal{X}^\epsilon$.
Define the contour $\Gamma'$ by 
$$\Gamma' = \hat{\Gamma} \setminus (\Gamma \cup \hat{\mathcal{X}}^\epsilon\cup \partial \mathcal{D}).$$

\begin{lemma}\label{whatlemma}
Let $\hat{w} = \hat{v}-I$. The following estimates hold uniformly for $t \geq 2$ and $\zeta \in \mathcal{I}$:
\begin{subequations}\label{hatwestimate}
\begin{align}\label{hatwestimate1}
& \|(1+|\cdot|)\partial_{x}^{l}\hat{w}\|_{(L^1\cap L^\infty)(\Gamma)} \leq \frac{C}{k_0 t}, \qquad
	\\\label{hatwestimate2}
& \|(1+|\cdot|)\partial_{x}^{l}\hat{w}\|_{(L^1 \cap L^{\infty})(\Gamma')} \leq C e^{-ct},
	\\\label{hatwestimate3}
& \|\partial_{x}^{l}\hat{w}\|_{(L^1\cap L^{\infty})(\partial \mathcal{D})} \leq C t^{-1/2},	\\\label{hatwestimate4}
& \|\partial_{x}^{l}\hat{w}\|_{L^1(\hat{\mathcal{X}}^\epsilon)} \leq C t^{-1}\ln t,
	\\\label{hatwestimate5}
& \|\partial_{x}^{l}\hat{w}\|_{L^\infty(\hat{\mathcal{X}}^\epsilon)} \leq C t^{-1/2}\ln t,
\end{align}
\end{subequations}
with $l=0,1$.
\end{lemma}
\begin{proof}
Using that $\partial_{\zeta}^{l}m_{\pm}^{k_0}$ and its inverse are uniformly bounded for $k \in \hat{\Gamma} \cap \mathcal{D}$ and $l=0,1$, the estimate (\ref{hatwestimate1}) follows from Lemma \ref{v3lemma}.

The contour $\Gamma'$ consists of the set $(\cup_{j=1}^4 \Gamma_j^{(3)})\setminus\bar{\mathcal{D}}$ and the images of this set under the rotations $k \mapsto \omega k$ and $k\mapsto \omega^2 k$. We estimate the $L^1$ and $L^\infty$ norms of $(1+|\cdot|)\partial_{x}^{l}\hat{w}$ on $\Gamma_j^{(3)} \setminus \bar{\mathcal{D}}$ for $j = 1$; similar arguments apply when $j = 2,3,4$, and (\ref{hatwestimate2}) then follows by symmetry.
We parametrize $\Gamma_1^{(3)} \setminus \bar{\mathcal{D}}$ by $k = k_0 + ue^{\frac{\pi i}{4}}$, $u > k_0/2$. Only the $(21)$ and $(23)$ elements of $\hat{w}=v_{1}^{(3)}-I$ are nonzero. Using \eqref{estimate r2astar at infty}, \eqref{delta1bound}, and \eqref{rePhi12}, the $(23)$-entry can be estimated as
\begin{align*}
|\hat{w}_{23}(x,t, k_0 + ue^{\frac{\pi i}{4}})|
\leq C|r_{2,a}^{*}(\omega k)| e^{-t\Phi_{32}}
\leq C e^{-\frac{3t}{4}\Phi_{32}}
\leq C e^{-ct(k_0 + u)^2}.
\end{align*}
The analysis of $|\partial_{x}\hat{w}_{23}|$ is similar.
Using  \eqref{estimate r1ahatstar at k0}, \eqref{delta1bound}, and the identity
\begin{align}\label{rePhi23}
\re \Phi_{21}(\zeta, k_0 + ue^{\frac{\pi i}{4}}) = -\sqrt{3} u^2, \qquad u \geq 0,
\end{align}
the $(21)$-entry can be estimated as
\begin{align*}
|\hat{w}_{21}(x,t, k_0 + ue^{\frac{\pi i}{4}})|
\leq C|\hat{r}_{1,a}^{*}| e^{t\Phi_{21}}
\leq C e^{c  t  \Phi_{21}}
\leq C e^{-ctu^2}, \qquad u \geq 0.
\end{align*}
Using in addition \eqref{estimate r1ahatstar at infty} and \eqref{der x of delta1}, we conclude that $|\partial_{x}\hat{w}_{21}(x,t,k_{0}+ue^{\frac{\pi i}{4}})| \leq C e^{-c t u^{2}}$.
Hence
$$|\partial_{x}^{l}\hat{w}(x,t, k_{0}+ue^{\frac{\pi i}{4}})| \leq C e^{-ctu^2}, \qquad u>k_{0}/2, \ l = 0,1.$$
It follows that the $L^1$ and $L^\infty$ norms of $(1+|\cdot|)\partial_{x}^{l}\hat{w}$, $l=0,1$, are $O(e^{-ct})$ as $t \to \infty$ on $\Gamma_1^{(3)} \setminus \bar{\mathcal{D}}$. This proves (\ref{hatwestimate2}).

The estimates in (\ref{hatwestimate3}) are immediate from (\ref{mmodmuestimate2}).

For $k \in \mathcal{X}^\epsilon$, we have $\hat{w} = m_-^{k_0} (v^{(3)} - v^{k_0}) (m_+^{k_0})^{-1}$, so (\ref{hatwestimate4}) and (\ref{hatwestimate5}) follow from (\ref{v3vk0estimate}) combined with the fact that $\partial_{\zeta}^{l}m_{\pm}^{k_0}$ and its inverse are uniformly bounded for $k \in \hat{\Gamma} \cap \mathcal{D}$ and $l=0,1$. 
\end{proof}
For a function $h$ defined on $\hat{\Gamma}$, the Cauchy transform $\hat{\mathcal{C}}h$ is defined by
\begin{align*}
(\hat{\mathcal{C}}h)(z) = \frac{1}{2\pi i} \int_{\hat{\Gamma}} \frac{h(z')dz'}{z' - z}, \qquad z \in \C \setminus \hat{\Gamma}.
\end{align*}
If $h \in \dot{L}^3(\hat{\Gamma})$, then $\hat{\mathcal{C}}h \in \dot{E}^{3}(\C \setminus \hat{\Gamma})$, and the left and right nontangential boundary values of $\hat{\mathcal{C}}h$, which we denote by $\hat{\mathcal{C}}_+ h$ and $\hat{\mathcal{C}}_- h$ respectively, exist a.e. on $\hat{\Gamma}$ and belong to $\dot{L}^3(\hat{\Gamma})$; furthermore, $\hat{\mathcal{C}}_\pm \in \mathcal{B}(\dot{L}^3(\hat{\Gamma}))$ and $\hat{\mathcal{C}}_+ - \hat{\mathcal{C}}_- = I$, where $\mathcal{B}(\dot{L}^3(\hat{\Gamma}))$ denotes the space of bounded linear operators on $\dot{L}^3(\hat{\Gamma})$, see \cite[Theorems 4.1 and 4.2]{LenellsCarleson}. 

The estimates in Lemma \ref{whatlemma} show that
\begin{align}\label{hatwLinfty}
\begin{cases}
\|(1+|\cdot|)\partial_{x}^{l}\hat{w}\|_{L^1(\hat{\Gamma})}\leq C t^{-1/2},
	\\
\|(1+|\cdot|)\partial_{x}^{l}\hat{w}\|_{L^\infty(\hat{\Gamma})}\leq C t^{-1/2}\ln t,
\end{cases}	 \qquad t \geq 2, \ \zeta \in \mathcal{I}, \ l = 0,1,
\end{align}
and hence, employing the general identity $\| f \|_{L^p} \leq \| f \|_{L^1}^{1/p}\|f \|_{L^{\infty}}^{(p-1)/p}$,
\begin{align}\label{Lp norm of what}
& \|(1+|\cdot|)\partial_{x}^{l}\hat{w}\|_{L^p(\hat{\Gamma})} 
\leq C t^{-\frac{1}{2}} (\ln t)^{\frac{p-1}{p}},  \qquad t \geq 2, \ \zeta \in \mathcal{I}, \ l = 0,1,
\end{align}
for each $1 \leq p \leq \infty$. The estimates \eqref{Lp norm of what} imply that
$\hat{w} \in \dot{L}^3(\hat{\Gamma}) \cap L^{\infty}(\hat{\Gamma})$.
We define $\hat{\mathcal{C}}_{\hat{w}}=\hat{\mathcal{C}}_{\hat{w}(x,t,\cdot)}: \dot{L}^{3}(\hat{\Gamma})+L^{\infty}(\hat{\Gamma}) \to \dot{L}^{3}(\hat{\Gamma})$ by $\hat{\mathcal{C}}_{\hat{w}}h := \hat{\mathcal{C}}_{-}(h \hat{w})$.

\begin{lemma}\label{invertiblelemma}
There exists a $T > 0$ such that $I - \hat{\mathcal{C}}_{\hat{w}(x, t, \cdot)} \in \mathcal{B}(\dot{L}^{3}(\hat{\Gamma}))$ is invertible whenever $t \geq T$ and $\zeta \in \mathcal{I}$.
\end{lemma}
\begin{proof}
Let $K:=\|\hat{\mathcal{C}}_{-}\|_{\mathcal{B}(\dot{L}^3(\hat{\Gamma}))}$. For each $h \in \dot{L}^{3}(\hat{\Gamma})$, we have $\|\hat{\mathcal{C}}_{\hat{w}}h\|_{\dot{L}^{3}(\hat{\Gamma})} \leq K\|\hat{w}\|_{L^\infty(\hat{\Gamma})} \|h\|_{\dot{L}^{3}(\hat{\Gamma})}$,
 and thus $\|\hat{\mathcal{C}}_{\hat{w}}\|_{\mathcal{B}(\dot{L}^{3}(\hat{\Gamma}))} \leq K \|\hat{w}\|_{L^\infty(\hat{\Gamma})}$.
By \eqref{hatwLinfty}, there exists a $T > 0$ such that $\|\hat{w}\|_{L^\infty(\hat{\Gamma})}<K^{-1}$ for $t \geq T$.
\end{proof}
In view of Lemma \ref{invertiblelemma}, we may define $\hat{\mu}(x, t, k)$ for $k \in \hat{\Gamma}$, $t \geq T$, and $\zeta = \frac{x}{t} \in \mathcal{I}$ by
\begin{align}\label{hatmudef}
\hat{\mu} = I + (I - \hat{\mathcal{C}}_{\hat{w}})^{-1}\hat{\mathcal{C}}_{\hat{w}}I \in I + \dot{L}^{3}(\hat{\Gamma}).
\end{align}
\begin{lemma}\label{lemma:mhat exists}
For $t \geq T$ and $\zeta \in \mathcal{I}$, there exists a unique solution $\hat{m} \in I + \dot{E}^{3}(\mathbb{C}\setminus \hat{\Gamma})$ of RH problem \ref{RHm}. This solution is given by
\begin{align}\label{hatmrepresentation}
\hat{m}(x, t, k) = I + \hat{\mathcal{C}}(\hat{\mu}\hat{w}) = I + \frac{1}{2\pi i}\int_{\hat{\Gamma}} \hat{\mu}(x, t, s) \hat{w}(x, t, s) \frac{ds}{s - k}.
\end{align}
\end{lemma}
\begin{proof}
Since $\hat{w} \in \dot{L}^3(\hat{\Gamma}) \cap L^{\infty}(\hat{\Gamma})$, this follows from \cite[Proposition 5.8]{LenellsCarleson}.
\end{proof}
\begin{lemma}\label{lemma: estimate on mu}
Let $1 \leq p < \infty$. For all sufficiently large $t$, we have
\begin{align*}
& \|\partial_{x}^{l}(\hat{\mu} - I)\|_{L^p(\hat{\Gamma})} \leq  C t^{-\frac{1}{2}}(\ln t)^{\frac{p-1}{p}}, \qquad l=0,1, \quad \zeta \in \mathcal{I}.
\end{align*}
\end{lemma}
\begin{proof}
Let $K_{p} := \|\hat{\mathcal{C}}_{-}\|_{\mathcal{B}(L^p(\hat{\Gamma}))} < \infty$ and assume $t$ is so large that $\|\hat{w}\|_{L^\infty(\hat{\Gamma})}<K_{p}^{-1}$. Standard estimates using the Neumann series show that
\begin{align*}
\|\hat{\mu} - I\|_{L^p(\hat{\Gamma})} \leq
\sum_{j=1}^{\infty} \|\hat{\mathcal{C}}_{\hat{w}}\|_{\mathcal{B}(L^p(\hat{\Gamma}))}^{j-1}\|\hat{\mathcal{C}}_{\hat{w}}I\|_{L^{p}(\hat{\Gamma})}  
\leq \sum_{j=1}^{\infty} K_{p}^{j}\|\hat{w}\|_{L^\infty(\hat{\Gamma})}^{j-1} \|\hat{w}\|_{L^p(\hat{\Gamma})} = \frac{K_{p} \|\hat{w}\|_{L^p(\hat{\Gamma})}}{1-K_{p}\|\hat{w}\|_{L^\infty(\hat{\Gamma})}}.
\end{align*}
The claim for $l = 0$ now follows from \eqref{hatwLinfty} and \eqref{Lp norm of what}. Using that
\begin{align*}
\partial_{x}(\hat{\mu} - I) = \partial_{x} \sum_{j=1}^{\infty} (\hat{\mathcal{C}}_{\hat{w}})^{j}I = \sum_{j=1}^{\infty} \big[ (\partial_{x}\hat{\mathcal{C}}_{\hat{w}}) \hat{\mathcal{C}}_{\hat{w}} \cdots \hat{\mathcal{C}}_{\hat{w}} + \ldots + \hat{\mathcal{C}}_{\hat{w}} \cdots \hat{\mathcal{C}}_{\hat{w}}(\partial_{x}\hat{\mathcal{C}}_{\hat{w}}) \big]I,
\end{align*}
we find
\begin{align*}
 \|\partial_{x}&(\hat{\mu} - I)\|_{L^p(\hat{\Gamma})}  \leq  \sum_{j=2}^{\infty} (j-1)\|\hat{\mathcal{C}}_{\hat{w}}\|_{\mathcal{B}(L^p(\hat{\Gamma}))}^{j-2} \|\partial_{x}\hat{\mathcal{C}}_{\hat{w}}\|_{\mathcal{B}(L^p(\hat{\Gamma}))} \|\hat{\mathcal{C}}_{\hat{w}}I\|_{L^{p}(\hat{\Gamma})}  
 	\\
& + \sum_{j=1}^{\infty} \|\hat{\mathcal{C}}_{\hat{w}}\|_{\mathcal{B}(L^p(\hat{\Gamma}))}^{j-1} \|\partial_{x}\hat{\mathcal{C}}_{\hat{w}}I\|_{L^{p}(\hat{\Gamma})} \\
\leq &\; C \sum_{j=2}^{\infty} j K_{p}^{j-2}\|\hat{w}\|_{L^\infty(\hat{\Gamma})}^{j-2} \|\partial_{x}\hat{w}\|_{L^\infty(\hat{\Gamma})} \|\hat{w}\|_{L^p(\hat{\Gamma})} + \sum_{j=1}^{\infty} K_{p}^{j}\|\hat{w}\|_{L^\infty(\hat{\Gamma})}^{j-1} \|\partial_{x}\hat{w}\|_{L^p(\hat{\Gamma})} 
	\\
\leq &\; C\frac{\|\partial_{x}\hat{w}\|_{L^\infty(\hat{\Gamma})} \|\hat{w}\|_{L^p(\hat{\Gamma})} + \|\partial_{x}\hat{w}\|_{L^p(\hat{\Gamma})}}{1-K_{p}\|\hat{w}\|_{L^\infty(\hat{\Gamma})}}
\end{align*}
and the claim for $l = 1$ follows from another application of \eqref{hatwLinfty} and \eqref{Lp norm of what}.
\end{proof}

\subsection{Asymptotics of $\hat{m}$}
The following nontangential limit exists as $k \to \infty$:
\begin{align*}
L(x,t):=\ntlim_{k\to \infty} k(\hat{m}(x,t,k) - I)
& = - \frac{1}{2\pi i}\int_{\hat{\Gamma}} \hat{\mu}(x,t,k) \hat{w}(x,t,k) dk.
\end{align*}
\begin{lemma}
As $t \to \infty$, 
\begin{align}\label{limlhatm}
L(x,t) = -\frac{1}{2\pi i}\int_{\partial \mathcal{D}} \hat{w}(x,t,k) dk + O(t^{-1}\ln t)
\end{align}
and (\ref{limlhatm}) can be differentiated termwise with respect to $x$ without increasing the error term.
\end{lemma}
\begin{proof}
Since
$$L(x,t) = -\frac{1}{2\pi i}\int_{\partial \mathcal{D}} \hat{w}(x,t,k) dk + L_1(x,t) + L_2(x,t),$$
where
\begin{align*}
& L_1(x,t) = -\frac{1}{2\pi i}\int_{\hat{\Gamma}\setminus\partial \mathcal{D}} \hat{w}(x,t,k) dk, \quad L_2(x,t) = -\frac{1}{2\pi i}\int_{\hat{\Gamma}} (\hat{\mu}(x,t,k)-I)\hat{w}(x,t,k) dk,
\end{align*}
the lemma follows from Lemmas \ref{whatlemma} and \ref{lemma: estimate on mu} and straightforward estimates.
\end{proof}
We infer from \eqref{mmodmuestimate1} that the function $F$ defined by
\begin{align*}
F(\zeta,t) = & - \frac{1}{2\pi i} \int_{\partial D_\epsilon(k_0)} \hat{w}(x,t,k) dk	
= - \frac{1}{2\pi i}\int_{\partial D_\epsilon(k_0)}((m^{k_0})^{-1} - I) dk
\end{align*}
satisfies
\begin{align*}
F(\zeta, t) = &\;  \frac{Y(\zeta,t) m_1^{X}(q(\zeta)) Y(\zeta,t)^{-1}}{3^{1/4}\sqrt{2}\sqrt{t}} + O( t^{-1} \ln t) \qquad \mbox{as } t \to \infty.
\end{align*}
The symmetry properties of $\hat{v}$ imply that both  $\mathcal{A} \hat{m}(x, t, \omega k) \mathcal{A}^{-1}$ and $\hat{m}(x, t, k)$ satisfy RH problem \ref{RHmhat}; by uniqueness they must be equal, i.e.,
$$\hat{m}(x, t, k) = \mathcal{A} \hat{m}(x, t, \omega k) \mathcal{A}^{-1}, \qquad k \in \C \setminus \hat{\Gamma}.$$
It follows that $\hat{\mu}$ and $\hat{w}$ also obey this symmetry. Using this in (\ref{limlhatm}), we find that the leading contribution from $\partial \mathcal{D}$ to the right-hand side of (\ref{limlhatm}) is
\begin{align*}
- \frac{1}{2\pi i}\int_{\partial \mathcal{D}} \hat{w}(x,t,k) dk
= & - \frac{1}{2\pi i}\bigg(\int_{\partial D_\epsilon(k_0)}+\int_{\omega \partial D_\epsilon(k_0)}+\int_{\omega^2 \partial D_\epsilon(k_0)}\bigg)  \hat{w}(x,t,k) dk
	\\
= &\; F(\zeta,t) + \omega \mathcal{A}^{-1} F(\zeta, t) \mathcal{A}
+ \omega^2 \mathcal{A}^{-2}F(\zeta,t)\mathcal{A}^2.
\end{align*}
Therefore, \eqref{limlhatm} implies that
\begin{align}\nonumber
\partial_{x}^{l}\lim_{k\to \infty} k(\hat{m}(x,t,k) - I)
= & \; \partial_{x}^{l}\bigg(\frac{ \sum_{j=0}^2 \omega^j \mathcal{A}^{-j} Y(\zeta,t) m_1^{X}(q(\zeta)) Y(\zeta,t)^{-1}\mathcal{A}^j}{3^{1/4}\sqrt{2}\sqrt{t}} \bigg)
	\\ \label{limlhatm2}
& + O(t^{-1}\ln t), \qquad t \to \infty, \quad l=0,1,
\end{align}
uniformly for $\zeta \in \mathcal{I}$.

\section{Asymptotics of $u(x,t)$}\label{uasymptoticssec}
Recall from the discussion in Section \ref{overviewsec} (see Proposition \ref{nprop} and Lemma \ref{nlemma}) that 
\begin{align*}
u(x,t) = -\frac{3}{2}\partial_{x} \Big(\lim_{k\to \infty}k(n_{3}(x,t,k) - 1)\Big),
\end{align*}
where $n = (\omega, \omega^2, 1)m$.
Taking the transformations of Section \ref{transsec} into account, we can write
$$m = \hat{m}H^{-1}\Delta^{-1}G^{-1}$$
for all $k \in \mathbb{C}\setminus \bar{\mathcal{D}}$, where $G$, $\Delta$, $H$ are defined in (\ref{Gdef}), (\ref{Deltadef}), and (\ref{Hdef}), respectively.
It follows from Lemma \ref{Glemma}, equation (\ref{Deltaatinfty}), and Lemma \ref{Hlemma} that
\begin{align}\label{recoverun}
u(x,t) = -\frac{3}{2}\frac{\partial}{\partial x}\lim_{k\to \infty}k(\hat{n}_{3}(x,t,k) - 1),
\end{align}
where $\hat{n} = (\omega, \omega^2, 1)\hat{m}$. Thus, utilizing (\ref{limlhatm2}) and the fact that $\overline{\Gamma(i\nu)} = \Gamma(-i\nu)$, 
\begin{align*}
u(x,t) = & -\frac{3}{2}\frac{d}{dx} \bigg(\begin{pmatrix}
\omega & \omega^2 & 1
\end{pmatrix} \frac{ \sum_{j=0}^2 \omega^j \mathcal{A}^{-j} Y(\zeta,t) m_1^{X}(q(\zeta)) Y(\zeta,t)^{-1}\mathcal{A}^j}{3^{1/4}\sqrt{2}\sqrt{t}} \bigg)_3 + O(t^{-1}\ln t)\\
= & -\frac{3}{2}\frac{d}{dx}\bigg( \frac{\omega^2d_0^{-1}e^{t\Phi_{21}(\zeta, k_0)}\beta_{21}+\omega d_0e^{-t\Phi_{21}(\zeta, k_0)}\beta_{12}}{3^{1/4}\sqrt{2}\sqrt{t}} \bigg)+ O(t^{-1}\ln t)\\
= & -\frac{3\times 2}{2\times3^{1/4}\sqrt{2t}}\frac{d}{dx}\re\bigg(\omega^2d_0^{-1}e^{t\Phi_{21}(\zeta, k_0)}\beta_{21} \bigg) + O(t^{-1}\ln t), \qquad t \to \infty.
\end{align*}
Using the identities
\begin{align*}
& |\Gamma(i\nu)| = \frac{\sqrt{2\pi}}{\sqrt{\nu}\sqrt{e^{\pi \nu}-e^{-\pi \nu}}} = \frac{\sqrt{2\pi}}{\sqrt{\nu}e^{\frac{\pi\nu}{2}}|q|}, \\
& \delta_{3}^{-1}(\zeta,k_{0})\delta_{5}^{-1}(\zeta,k_{0}) = \exp \bigg[ i \nu \log( 3k_{0}^{2}) + \frac{1}{\pi i}\int_{k_{0}}^{\infty} \log |\omega k_{0}-s|d \ln(1-|r_1(s)|^{2}) \bigg],
\end{align*}
we conclude that, as $t\to \infty$,
\begin{align*}
u(x,t) = & -\frac{3^{3/4}}{\sqrt{2t}} \frac{d}{dx}\re\bigg\{ \sqrt{\nu}\exp\bigg[ \frac{4\pi i}{3} + i\nu\ln(6\sqrt{3}tk_0^2)  -i \sqrt{3} k_0^2 t  \\
&  - \frac{1}{\pi i}\int_{k_0}^{\infty} \ln \frac{|s-k_0|}{|s-\omega k_0|} d\ln(1 - |r_1(s)|^2) + i \Big(\frac{\pi}{4}-\arg{q}-\arg\Gamma(i\nu)\Big)\bigg]\bigg\} + O(t^{-1}\ln t)\\
= & -\frac{3^{3/4}\sqrt{\nu}}{\sqrt{2t}} \frac{d}{dx}\cos\bigg(\frac{19\pi}{12}+\nu\ln(6\sqrt{3}tk_0^2)-\sqrt{3}k_0^2t
-\arg{q}\\
&-\arg{\Gamma(i\nu)} + \frac{1}{\pi }\int_{k_0}^{\infty} \ln \frac{|s-k_0|}{|s-\omega k_0|} d\ln(1 - |r_1(s)|^2)\bigg)
 + O(t^{-1}\ln t)\\
= & -\frac{3^{5/4}k_0\sqrt{\nu}}{\sqrt{2t}}  \sin\bigg(\frac{19\pi}{12}+\nu\ln(6\sqrt{3}tk_0^2)-\sqrt{3}k_0^2t
-\arg{q} 
	\\
& - \arg{\Gamma(i\nu)} + \frac{1}{\pi }\int_{k_0}^{\infty} \ln \frac{|s-k_0|}{|s-\omega k_0|} d\ln(1 - |r_1(s)|^2)\bigg) + O(t^{-1}\ln t)
\end{align*}
uniformly for $\zeta \in \mathcal{I}$. 
This proves \eqref{uasymptotics} and completes the proof of Theorem \ref{mainth}.

\appendix
\section{Exact solution on a cross}\label{appA}
\renewcommand{\theequation}{A.\arabic{equation}}
Let $X = X_1 \cup \cdots \cup X_4 \subset \C$ be the cross defined by
\begin{align} \nonumber
&X_1 = \bigl\{se^{\frac{i\pi}{4}}\, \big| \, 0 \leq s < \infty\bigr\}, &&
X_2 = \bigl\{se^{\frac{3i\pi}{4}}\, \big| \, 0 \leq s < \infty\bigr\},
	\\ \label{Xdef}
&X_3 = \bigl\{se^{-\frac{3i\pi}{4}}\, \big| \, 0 \leq s < \infty\bigr\}, &&
X_4 = \bigl\{se^{-\frac{i\pi}{4}}\, \big| \, 0 \leq s < \infty\bigr\},
\end{align}
and oriented away from the origin, see Figure \ref{X.pdf}. Let $\D \subset \C$ denote the open unit disk and define the function $\nu:\D \to (0,\infty)$ by
$\nu(q) = -\frac{1}{2\pi} \ln(1 - |q|^2)$.
We consider the following family of RH problems parametrized by $q \in \D$.

\begin{figure}
\begin{center}
 \begin{overpic}[width=.4\textwidth]{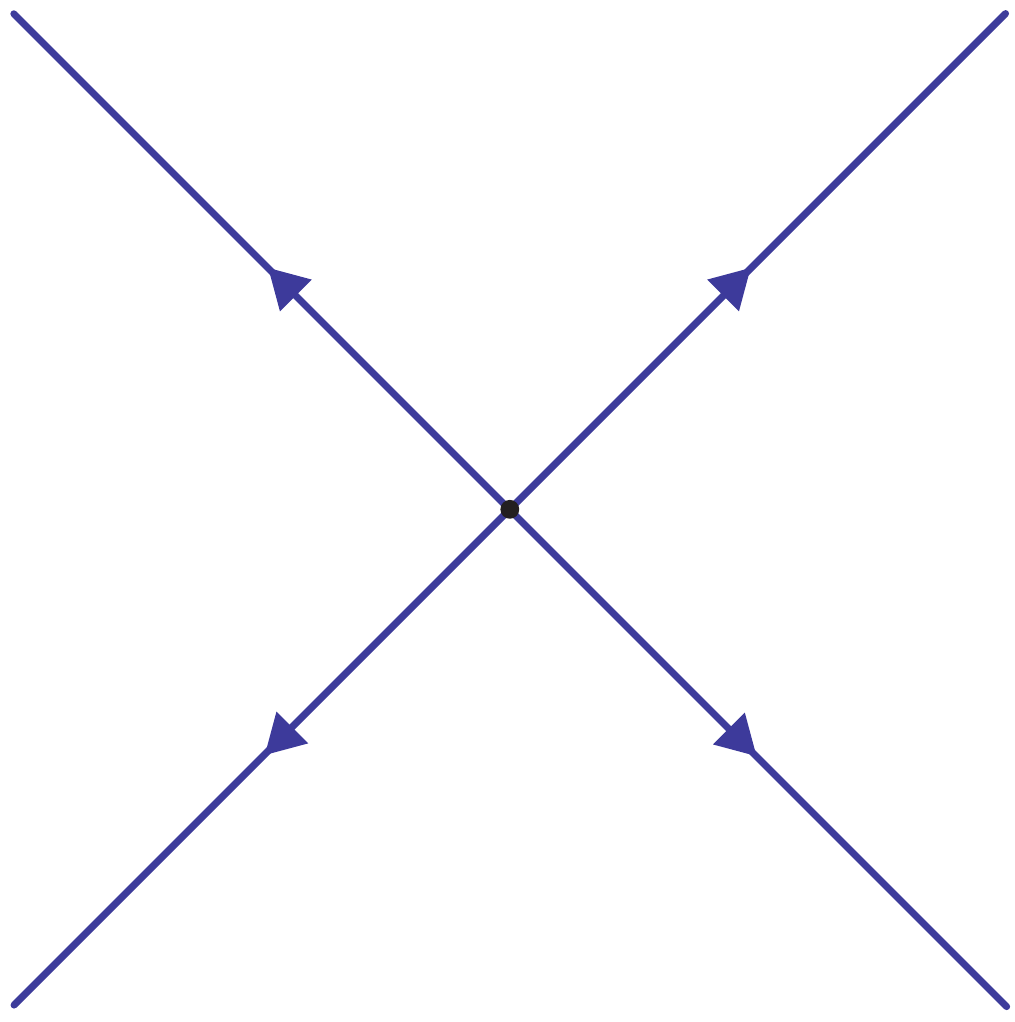}
      \put(73.5,68){\small $X_1$}
      \put(18.5,68){\small $X_2$}
      \put(17,28){\small $X_3$}
      \put(75,28){\small $X_4$}
      \put(48.3,42.6){$0$}
    \end{overpic}
     \begin{figuretext}\label{X.pdf}
        The contour $X = X_1 \cup X_2 \cup X_3 \cup X_4$ defined in (\ref{Xdef}).
     \end{figuretext}
     \end{center}
\end{figure}

\begin{RHproblem}[RH problem for $m^X$]\label{RHmc}
Find a $3 \times 3$-matrix valued function $m^X(q, z)$ with the following properties:
\begin{enumerate}[$(a)$]
\item $m^X(q, \cdot) : \C \setminus X \to \mathbb{C}^{3 \times 3}$ is analytic.

\item The limits of $m^X(q, z)$ as $z$ approaches $X \setminus \{0\}$ from the left and right exist, are continuous on $X \setminus \{0\}$, and related by
\begin{align*}
m_+^X(q, z) =  m_-^X(q, z) v^X(q, z), \qquad k \in X\setminus\{0\},
\end{align*}
where the jump matrix $v^X(q, z)$ is defined by
\begin{align}
& \begin{pmatrix} 1 & 0 & 0	\\
\frac{\bar{q}}{1 - |q|^2} z^{-2i\nu(q)}e^{\frac{iz^2}{2}} & 1 & 0 \\
0 & 0 & 1 \end{pmatrix} \mbox{ if } z \in X_{1}, & & \begin{pmatrix} 1 & q z^{2i\nu(q)} e^{-\frac{iz^2}{2}} & 0 	\\
0 & 1 & 0 \\
0 & 0 & 1 \end{pmatrix} \mbox{ if } z \in X_{2}, \nonumber \\
& \begin{pmatrix} 1 & 0 & 0 \\
-\bar{q} z^{-2i\nu(q)}e^{\frac{iz^2}{2}} & 1 &  0 \\
0 & 0 & 1 \end{pmatrix} \mbox{ if } z \in X_{3}, & & \begin{pmatrix} 1 &  \frac{-q}{1 - |q|^2}z^{2i\nu(q)} e^{-\frac{iz^2}{2}} & 0 \\
 0 & 1 & 0 	\\
0 & 0 & 1 \end{pmatrix} \mbox{ if } z \in X_{4}, \label{vXdef}
\end{align}
with the branch cut running along the positive real axis, i.e., $z^{2i\nu(q)} = e^{2i\nu(q)\ln_0(z)}$.

\item $m^{X}(q,z) = I + O(z^{-1})$ as $z \to \infty$.

\item $m^{X}(q,z) = O(1)$ as $z \to 0$.
\end{enumerate}
\end{RHproblem}

The proof of the following lemma is standard and relies on deriving an explicit formula for the solution $m^X$ in terms of parabolic cylinder functions \cite{I1981}.

\begin{lemma}[The solution $m^X$]\label{Xlemma}
  The RH problem \ref{RHmc} has a unique solution $m^X(q, z)$ for each $q \in \D$. This solution satisfies
\begin{align}\label{mXasymptotics}
  m^X(q, z) = I + \frac{m_1^X(q)}{z} + O\biggl(\frac{1}{z^2}\biggr), \qquad z \to \infty,~~q \in \D,
\end{align}
where the error term is uniform with respect to $\arg z \in [0, 2\pi]$ and $q$ in compact subsets of $\D$, and the function $m_1^X(q)$ is defined by
\begin{align}\label{m1Xdef}
m_1^X(q) =\begin{pmatrix} 
0 & \beta_{12} & 0 	\\
\beta_{21} & 0 & 0 \\
0 & 0 & 0 \end{pmatrix}, \qquad q \in \D,
\end{align}
where $\beta_{12}$ and $\beta_{21}$ are defined by
$$
\beta_{12}=\frac{\sqrt{2\pi}e^{-\frac{\pi i}{4}} e^{-\frac{5\pi \nu}{2}} }{\bar{q} \Gamma(-i\nu)} ,\qquad
\beta_{21}=\frac{\sqrt{2\pi}e^{\frac{\pi i}{4}} e^{\frac{3\pi \nu}{2}} }{q \Gamma(i\nu)},
\qquad q \in \D.
$$
Moreover, for each compact subset $K$ of $\D$,
$$\sup_{q \in K} \sup_{z \in \C \setminus X} |\partial_{q}^{l}m^X(q, z)| < \infty, \qquad l=0,1.$$
\end{lemma}

\subsection*{Acknowledgements}
CC acknowledges support from  the European Research Council, Grant Agreement No. 682537 and the Swedish Research Council, Grant No. 2015-05430.
JL acknowledges support from the European Research Council, Grant Agreement No. 682537, the Swedish Research Council, Grant No. 2015-05430, the G\"oran Gustafsson Foundation, and the Ruth and Nils-Erik Stenb\"ack Foundation.
DW acknowledges support from the Swedish Research Council, Grant No. 2015-05430, the National Natural Science Foundation of China, Grant No. 11971067, and the Beijing Great Wall Talents Cultivation Program, Grant No. CIT\&TCD20180325.

\bibliographystyle{plain}
\bibliography{is}

\end{document}